\documentclass[letterpaper,11pt]{amsart}

\usepackage{amsfonts,amsthm,latexsym,amsmath,amssymb,amscd,amsmath, mathrsfs, xypic,geometry,enumerate}
\usepackage{color}
\usepackage[utf8]{inputenc}
\usepackage{picture}
\usepackage{tikz}
\usepackage{graphicx}
\usepackage{verbatim}

\newtheorem{theorem}{Theorem}[section]
\newtheorem{proposition}[theorem]{Proposition}
\newtheorem{lemma}[theorem]{Lemma}
\newtheorem{corollary}[theorem]{Corollary}
\newtheorem{conjecture}{Conjecture}

\theoremstyle{definition}

\newtheorem{definition}[theorem]{Definition}
\newtheorem{example}[theorem]{Example}
\newtheorem{remark}[theorem]{Remark}

\title{Waring loci and the Strassen conjecture}

\author[E. Carlini]{Enrico Carlini}
\address[E. Carlini]{DISMA- Department of Mathematical Sciences, Politecnico di Torino, Turin, Italy}
\email{enrico.carlini@polito.it}

\author[M. V. Catalisano]{Maria Virginia Catalisano}
\address[M. V. Catalisano]{Dipartimento di Ingegneria Meccanica, Energetica, Gestionale e dei
Trasporti, Universita degli studi di Genova, Genoa, Italy}
\email{catalisano@diptem.unige.it}

\author[A. Oneto]{Alessandro Oneto}
\address[A. Oneto]{INRIA Sophia Antipolis M\'editerran\'ee (team Aromath), Sophia Antipolis, France}
\email{alessandro.oneto@inria.fr}

\newcommand{\rk}{\mathrm{rk}}

\newcommand{\CC}{\mathbb{C}}
\newcommand{\PP}{\mathbb{P}}
\newcommand{\XX}{\mathbb{X}}

\newcommand{\caC}{\mathcal{C}}

\newcommand{\caL}{\mathcal{L}}

\newcommand{\caW}{\mathcal{W}}
\newcommand{\caF}{\mathcal{F}}

\begin{document}
\setlength{\parindent}{5pt}

\begin{abstract}
The Waring locus of a form $F$ is the collection of the degree one forms appearing in some minimal sum of powers decomposition of $F$. In this paper, we give a complete description of Waring loci for several family of forms, such as quadrics, monomials, binary forms and plane cubics. We also introduce a Waring loci version of Strassen's Conjecture, which implies the original conjecture, and we prove it in many cases.
\end{abstract}

\maketitle

\section{Introduction}
Let $S = \CC[x_0,\ldots,x_n]=\bigoplus_{i \geq 0}S_i$ be the standard graded
polynomial ring. An element in $S_i$, that is a degree $i$ homogeneous polynomial, is called a {\it form}.

A {\em Waring decomposition}, also called a {\em sum of powers decomposition}, of $F\in S_d$ is an expression of the form
\[F=L_1^d+\ldots+L_r^d,\]
for linear forms $L_i\in S_1$. The \emph{Waring rank}, or simply \emph{rank}, of $F$ is
\[\rk(F)=\min\{r:F=L_1^d+\ldots+L_r^d, ~L_i\in S_1\mbox{ for } 1\leq i\leq r\},\]
and we say that a Waring decomposition of $F$ is {\it minimal} if it involes $\rk(F)$ summands.

In the last decades, an intense research activity focused on computing Waring
ranks and (minimal) Waring decompositions of homogeneous polynomials. A celebrated result in this field is due to J.Alexander and A.Hirschowitz who determined the rank of a generic form \cite{AH}, but still very little is known for given specific forms. The main abstract tools to deal with Waring ranks are in \cite{LT,CCCGW}.

One of the main reason for the great interest in Waring ranks and Waring decompositions is due to the relations with the theory of symmetric tensors
and their decompositions as sums of rank one tensors which have applications
in Algebraic Statistics, Biology, Quantum Information Theory and more, see \cite{L}.

In this paper we investigate minimal Waring decompositions. Of particular interest are the cases when
the minimal Waring decomposition is {\it unique}, called in the literature
the {\it identifiable} cases, see \cite{CC, Mel06, Mel09, BCO, COV, GM16}. Since finding a minimal Waring decomposition is usually beyond our computational capabilities, we decide to investigate properties of {\it all} minimal Waring decompositions of a given form at once.

\begin{definition}
 Given a form $F$, we define the {\it Waring locus} of $F$ as the set of linear forms that appear, up to scalar, in a
minimal Waring decomposition of $F$, namely
$$
\caW_F = \{[L]\in\PP(S_1) : \exists L_2,\ldots,L_r\in S_1,~ F \in \langle L^d,L_2^d,\ldots,L_r^d\rangle,~r =\rk(F)\}.
$$
We define the \emph{locus of forbidden points} as its complement,
$
\caF_F=\PP(S_1)\setminus\caW_F.
$
\end{definition}

In \cite{BC}, the authors suggested the importance of the study of the Waring loci of homogeneous
polynomials. In particular, they proved that $\caW_F$ has no isolated points if $F\in S_d$ is not identifiable and if $\rk(F)< \frac{3d}{2}$.

One of our main result is a description of $\caW_F$, or equivalently of $\caF_F$, in the following cases:
\begin{enumerate}
 \item quadrics, that is degree two forms, see Corollary \ref{Quadrics Corollary};
 \item monomials, see Theorem \ref{monomials Thm};
 \item binary forms, that is forms in two variable, see Theorem \ref{binaryTHM};
 \item plane cubics, that is degree three forms in three variables, see Section \ref{Plane Cubics Section};
 \item sums of particular families of homogeneous polynomials in different set of variables, see Theorem \ref{waringdegTHM} and Section \ref{Furtherresultsandopenproblems}.
\end{enumerate}

In Section \ref{Basic Section}, we introduce the basics facts and our main tool: the {\it Apolarity Lemma}, Lemma \ref{Apolarity Lemma}.
The Apolarity Lemma provides a very explicit recipe to find Waring decompositions of an homogeneous
polynomials $F$. In particular, it states that Waring decompositions of $F$ corresponds to ideals of reduced points contained in the ideal $F^{\perp}$,
namely the ideal of polynomials annihilating $F$ by acting as differentials. The reason why we succeeded in finding Waring loci in the cases listed above is that those are the cases when we can give a very precise description of all the possible minimal set of reduced points contained in the annihilating ideals.

In Section \ref{Results Section}, we present our main results about Waring loci.

In Section \ref{Strassen's conjecture}, we discuss relations with Strassen's conjecture.

In Section \ref{Furtherresultsandopenproblems}, we present some more technical results and open problems.

\smallskip
{\bf Acknowledgment.} The authors want to thank Jaros\l{}aw Buczy\'nski and Brian Harbourne for comments on an earlier version of the paper. They are also very grateful to the anonymous referee for her/his constructive criticism. The third author wants to thank the first author and the School of Mathematical Sciences at Monash University (Australia) for the hospitality during a visit of three months while the work has been mostly done. The visit has been partially supported by {\it G~S~Magnuson Foundation} from {\it Kungliga Vetenskapsakademien} (Sweden). The first and second author were partially supported by GNSAGA of INdAM and by MIUR  funds (Italy).

\section{Basics}\label{Basic Section}

We introduce some basic notions on {\it Apolarity theory}, see also \cite{IK, G}.

We consider two polynomial rings $S=\CC[x_0,\ldots,x_n]=\bigoplus_{i\geq 0} S_i$ and $T=\CC[X_0,\ldots,X_n]=\bigoplus_{i\geq 0} T_i$ with standard gradation,
where $S$ has the structure of a $T$-module via differentiation; namely, we consider the \emph{apolarity action} given by
$$g\circ F = g(\partial_{x_0},\ldots,\partial_{x_n})F,~~ \text{ for }g\in T,~F\in S.$$

Given $F\in S_d$ we define the {\it apolar ideal} of $F$ as
\[F^\perp=\{\partial\in T: \partial\circ F=0\}.\]

We say that $F\in \CC[x_0,\ldots,x_n]$ \emph{essentially involves} $t+1$ variables if $\dim (F^\perp)_1=n-t$. In other words, if $F$ essentially involves $t+1$ variables, there exist linear forms $l_0,\ldots,l_t\in \CC[x_0,\ldots,x_n]$ such that $F\in\mathbb{C}[l_0,\ldots,l_t]$, see \cite{Creducing}.

We are interested in describing the minimal Waring decompositions of a form $F\in S_d$ and our main tool is the following.

\begin{lemma}[{\sc Apolarity Lemma}]\label{Apolarity Lemma}
 Let $\XX = \{P_1,\ldots,P_s\}\subset\PP^n$ be a set of reduced points where $P_i=[p_{0}^{(i)}:\ldots:p_n^{(i)}]$, for all $i = 1,\ldots,s$, and let $L_i = p_{0}^{(i)}x_0+\ldots+p_n^{(i)}x_n$, for all $i = 1,\ldots,s$. Then, if $F$ is a degree $d$ homogeneous polynomial, the following are equivalent:
 \begin{enumerate}
  \item $I_{\XX}\subset F^\perp$;
  \item $F = c_1L_{1}^d+\ldots+c_sL_{s}$, for  $c_1,\ldots,c_s\in\CC$.
 \end{enumerate}
\end{lemma}

A set of points $\mathbb{X}$ such that the conditions of Apolarity Lemma hold is said to be \emph{apolar to $F$}.


\begin{example}
 Consider the monomial $M = xyz\in\CC[x,y,z]$. It is easy to check that $M^{\perp}=(X^2,Y^2,Z^2)$. Hence the
 ideal $I = (X^2-Y^2,X^2-Z^2)\subset F^\perp$
 corresponds to the four reduced points $[1:\pm 1:\pm 1]$ and we have the Waring decomposition
 $$M = \frac{1}{24}\left[(x+y+z)^3 - (x-y+z)^3 - (x+y-z)^3 + (x-y-z)^3\right].$$
\end{example}

We can describe the Waring locus of a form $F$ in terms of the apolar points to $F$, namely
$$
\caW_F = \{P\in\PP^n : P\in\mathbb{X},~I_\mathbb{X}\subset F^\perp\mbox{ and } |\mathbb{X}|=\rk(F)\}.
$$

The following result, also given in \cite{BL} in the case of tensors, allows us to study a form $F$ in the ring of polynomials with the smallest number of variables.
In particular, we want to show that, if $F\in\CC[y_0,\ldots,y_m]$ essentially involves $n+1$ variables and $\mathbb{X}$ is a minimal
set of points apolar to $F$, then $\mathbb{X}\subset\PP^m$ is contained in a $n$-dimensional linear subspace of $\PP^m$. Hence,
$\caW_F\subset\PP^n$ contains all points belonging to any minimal set of points apolar to $F$.

\begin{proposition}\label{prop:essential_variables}
Let $F\in \mathbb{C}[x_0,\ldots,x_n, x_{n+1}, \ldots, x_{m}]$ be a degree $d$ form such that $(F^\perp)_1=(X_{n+1},\ldots,X_{m})$. If
\[F=\sum_1^r L_i^d\]
where $r=\rk(F)$ and the $L_i$ are linear forms in $\mathbb{C}[x_0,\ldots,x_n, x_{n+1}, \ldots, x_{m}]$, then
\[
L_i\in \mathbb{C}[x_0,\ldots,x_n]\subset \mathbb{C}[x_0,\ldots,x_n, x_{n+1}, \ldots, x_{m}]
\]
for all $i,1\leq i\leq r$.
\end{proposition}
\begin{proof} 
We proceed by contradiction. Assume that $L_1=x_{n+1}+\sum_{i\neq n+1} a_i x_i$, that is to assume that $L_1$  actually
involves the variable $x_{n+1}$. By assumption $\rk(F-L_1^d)< r=\rk(F)$. However, since $L_1$ is linearly independent with
$x_1,\ldots,x_n$ we can apply the following fact (see \cite[Proposition 3.1]{CCC}): if $y$ is a new variable, then
\[\rk(F+y^d)=\rk(F)+1.\]
Hence, $\rk(F-L_1^d)=\rk(F)+1$ and this is a contradiction.
\end{proof}

\begin{remark}\label{essential variables}
 Using the previous result, performing a linear change of variables, and restricting the ring, we may always assume
 that $F\in S_d$ essentially involves $n+1$ variables; hence, we always see $\caW_F$ and $\caF_F$ as subsets of $\PP^n$.
\end{remark}

\section{Waring loci}\label{Results Section}
In this section, we give our results about $\caW_F$ and $\caF_F$.

\subsection{Quadrics}

We begin with the study of elements of $S_2$, i.e. quadrics in $\PP^n$. We recall that to each quadric $Q$ we can associate a symmetric
$(n+1)\times (n+1)$ matrix $A_Q$ and that $\rk(Q)$ equals the rank of $A_Q$.

\begin{proposition}\label{Quadrics Proposition}
If $Q=x_0^2+\ldots+x_n^2$, then $\caF_Q=V(\check Q)\subset\PP^n$, where $\check Q=X_0^2+\ldots+X_n^2\in T_2$.
\end{proposition}

\begin{proof}
A point $P=[a_0:\ldots:a_n]$ is a forbidden point for $Q$
if and only if
\[\rk(Q-\lambda L_P^2)=n+1, \mbox{ for all }\lambda\in\mathbb{C}\]
where $L_P=\sum_0^n a_ix_i$. Thus, $P$ is a forbidden point for $Q$ if and only if the symmetric
matrix corresponding to the quadratic form $Q-\lambda L_P^2$ has non-zero determinant for all $\lambda\in\mathbb{C}$. Thus, $P$ is a forbidden point if and only if the symmetric
matrix $A_{L^2}$ corresponding to $L^2$ only have zero eigenvalues (note that, over $\mathbb{C}$, $A_{L^2}$ is not necessarily similar to a diagonal matrix).

We now prove that $A_{L^2}$ only have zero eigenvalues if and only if $\sum_0^n a_i^2=0$.

Assume that zero is the only eigenvalue of $A_{L^2}$. Let $\mathbf{a}=(\begin{array}{ccc}
a_0 & \ldots & a_n
\end{array})$ and  note that
\[\mathbf{a} A_{L^2} = (a_0^2+\ldots+a_n^2)\mathbf{a},
\]
thus $\sum_0^n a_i^2=0$ since it is an eigenvalue.

Now assume that $\sum_0^n a_i^2=0$.
Note that
\[ A_{L^2} ^2 = \left(\sum_0^n a_i^2\right) A_{L^2}. \]
thus $A_{L^2}^2=0$ and hence zero is the only eigenvalue.

 Hence, $P$ is a forbidden point if and only if $\sum_0^n a_i^2=0$ and the proof is now completed.
\end{proof}

\begin{corollary}\label{Quadrics Corollary}

Let $Q(x_0, \ldots, x_n) \in S_2$ be a rank $n+1$ quadric and let $B$ be an $(n+1)\times (n+1)$ matrix such that the change of variables
$$ (x_0, \ldots, x_n) =  (y_0, \ldots, y_n) B$$
gives $Q=y_0^2+\ldots+y_n^2$. Then, $\caF_Q=V(\check Q)\subset\PP^n$, where $\check Q(X_0,\ldots,X_n)$ is the quadratic form
\[(X_0 \ldots X_n) B^t B (X_0 \ldots X_n)^t.\]
\end{corollary}
\begin{proof}

Let $P=[a_0:\ldots:a_n]$ and  $L_P=\sum_0^n a_ix_i = (a_0, \ldots , a_n) (x_0, \ldots , x_n)^t$.

By the the linear change of variables $ (x_0, \ldots, x_n)=  (y_0, \ldots, y_n) B$, we get
$$ L_P =  (a_0, \ldots , a_n) B^t (y_0, \ldots, y_n)^t .
$$
Let  $ (a_0, \ldots , a_n) B^t  =  (b_0, \ldots , b_n) $, so that
$ L_P =   b_0y_0 + \cdots + b_ny_n $.
Using Proposition \ref{Quadrics Proposition} we know that a point
  $[b_0:\ldots:b_n]\in\caF_{Q(y_0,\dots,y_n)}$ if and only if
 $\sum b_i^2=0$. That is,
 $$(a_0, \ldots , a_n) B^t B (a_0, \ldots , a_n)^t =0.
 $$
 and the result follows.
\end{proof}

\subsection{Monomials}\label{Monomials Section}

In this section, we consider monomials $x_0^{d_0}\ldots x_n^{d_n}\in\CC[x_0,\ldots,x_n]$. We always assume that the exponents are increasingly ordered, that is $d_0 \leq …\leq d_n$.
In \cite{CCG}, the authors proved an explicit formula for the Waring rank of monomials, i.e.
$$\rk(x_0^{d_0}\ldots x_n^{d_n})=\frac{1}{d_0+1}\prod_{i=0}^n (d_i+1).$$

We also know from \cite{BBT} that minimal sets of apolar points to monomials are complete intersections, namely they are given by the intersection of $n$ hypersurfaces
in $\PP^n$ of degrees $d_1+1,\ldots,d_n+1$ intersecting properly. Moreover, it is known that, using our new terminology, the points lying on the hyperplane $\{X_i = 0\}$, for $i = 0,\ldots,m$, where $m = \max\{i ~|~ d_i = d_0\}$ are forbidden. This is implicitly proved during the proof of \cite[Proposition 3.1]{CCG} and can also be found in \cite[Corollary 19]{BBT}. Here, we prove that these are actually the only forbidden points of monomials.

\begin{theorem}\label{monomials Thm}
If $M=x_0^{d_0}\cdots x_n^{d_n}\in S$ , then
 $$\caF_M=V(X_0\cdots X_m) \subset \PP^n,$$
 where $m = \max\{i ~|~ d_i = d_0\}$.
\end{theorem}
\begin{proof}
The perp ideal of $M$ is $M^{\perp} = (X_0^{d_0+1},\ldots,X_n^{d_n+1})$.
  Consider any point $P=[p_0:\ldots:p_n]\notin V(X_0\cdots X_m)$, we may assume $p_0=1$ and we will prove that $P\in\caW_M$, that is $P\not\in\caF_M$. We construct the following hypersurfaces in $\PP^n$
  $$
  H_i=
  \begin{cases}
   X_i^{d_i+1}-p_i^{d_i+1}X_0^{d_i+1} & \text{ if }p_i\neq 0; \\ \\
   X_i^{d_i+1}-X_iX_0^{d_i} & \text{ if }p_i=0. \\
  \end{cases}
  $$
  Note that, for any
  $i=1,\ldots,n$, the hypersurface $H_i=0$ is the union of $d_i+1$ hyperplanes.

  The ideal $I=(H_1,\ldots,H_n)$ is contained in $M^{\perp}$ and $V(I)$ is the set of reduced points $[1:q_1:\ldots:q_n]$ where
  $$
   q_i \in
   \begin{cases}
    \{\xi_i^{j}p_i ~|~ j=0,\ldots,d_i\}, & \text{ if }p_i\neq 0,\text{ where }\xi_i^{d_i+1}=1; \\ \\

    \{\xi_i^{j} ~|~ j=0,\ldots,d_i-1\}\cup\{0\}, &\text{ if }p_i= 0,\text{ where }\xi_i^{d_i}=1. \\
   \end{cases}
  $$
  Thus, we have a set of $\rk(M)$ distinct points apolar to $M$ and containing the point $P$; hence, $P\in\caW_M$ and $V(X_0\cdots X_m) \supset \caF_M$.

To conclude the proof we need to prove that $$V(X_0\cdots X_m) \subset \caF_M$$ and this readily follows by \cite[Remark 3.3]{CCG} or by \cite[Corollary 19]{BBT}.
\end{proof}

\begin{remark}
 In the case $d_0\geq 2$, the second part of the proof can be explained as a direct consequence of the formula for the rank of monomials. Indeed, in the same notations
 as Theorem \ref{monomials Thm}, for any $i=1,\ldots,m$, we have that $\rk(M)=\rk(\partial_{x_i} \circ M)$. Therefore, given any minimal Waring decomposition of $M=\sum_{j=1}^r L_j^d$,
 by differentiating both sides, we must have $\partial_{x_i} \circ L_j \neq 0$ for all $i=1,\ldots,m$. Hence $[L_j] \notin V(X_0\cdots X_m)$.
\end{remark}

\subsection{Binary forms}

In this section we deal with the case $n=1$, that is the case of forms in two variables. The knowledge on the Waring rank of binary forms goes back to J.J. Sylvester
\cite{Syl}. It is known that, if $F\in \CC[x,y]_d$, then $F^{\perp} = (g_1,g_2)$ and $\deg(g_1)+\deg(g_2)=d+2$. Moreover,
if we assume $d_1=\deg(g_1)\leq d_2=\deg(g_2)$, then $\rk(F)=d_1$ if $g_1$ is square free and $\rk(F)=d_2$ otherwise. See \cite{CS} for more about the rank of binary
forms.

\begin{theorem}\label{binaryTHM}
Let $F$ be a degree $d$ binary form and let $g\in F^\perp$ be an element of minimal degree. Then,
\begin{enumerate}
\item if $\rk(F)<\lceil{d+1 \over 2}\rceil$, then $\caW_F=V(g)$;
\item if $\rk(F)>\lceil{d+1 \over 2}\rceil$, then $\caF_F=V(g)$;
\item if $\rk(F)=\lceil{d+1 \over 2}\rceil$ and $d$ is even, then $\caF_F$ is finite and not empty; \\
 if $\rk(F)=\lceil{d+1 \over 2}\rceil$ and $d$ is odd, then $\caW_F=V(g)$.
\end{enumerate}
\end{theorem}
\begin{proof}
(1) It is is enough to note that the decomposition of $F$ is unique and the unique apolar set of points is $V(g)$.

(2) As mentioned above, in this case we have that $F^\perp=(g_1,g_2)$, where $d_1=\deg(g_1)<\deg(g_2)=d_2$, $d_1+d_2=d+2$, $g_1$ is not square free, and $\rk(F)=d_2$. In particular, $g_1$ is an element of minimal degree in the apolar ideal.
We first show that $\caF_F\supseteq V(g_1)$. Let $P=V(l)\in V(g_1)$ for some linear form $l$, that is $l$ divides $g_1$. We want to show that there is no apolar set of points to $F$ containing $P$. Thus, it is enough to show that there is no square free element of degree $d_2$ in $F^\perp$ divisible by $l$. Since $g_1$ and $g_2$ have no common factors, and $l$ divides $g_1$, it follows that the only elements of degree $d_2$ in $F^\perp$ divisible by $l$ are multiples of $g_1$, thus they are not square free. Hence, $P\in\caF_F$.
We now prove that $\caF_F\subseteq V(g_1)$ by showing that, if $P=V(l)\not\in V(g_1)$, then $P\in\caW_F$. Note that $l$ does not divide $g_1$ and consider
\[F^\perp:(l)=(l\circ F)^\perp=(h_1,h_2)\]
where $c_1=\deg(h_1),c_2=\deg(h_2)$ and $c_1+c_2=d+1$. Since $h_1$ is a minimal degree element in $F^\perp$ and $l$ does not divide $g_1$, we have $h_1=g_1$ and
$c_2=d_2-1$. Thus $\rk(F)=\rk(l \circ F)+1$. Since $(F^\perp:(l))_{d_2-1}$ is base point free, we can choose $h\in F^\perp:(l)$ to be a degree $d_2-1$ square free element not divisible by $l$. Hence, $P\in V(lh)$ and $V(lh)$ is
a set of $d_2$ points apolar to $F$.

(3) Let $F^\perp=(g_1,g_2)$, $d_1=\deg(g_1)$, and $d_2=\deg(g_2)$. If $d$ is odd, then $d_2=d_1+1$ and $\rk(F)=d_1$; thus $g_1$ is a square free element of minimal
degree and $F$ has a unique apolar set of $d_1$ distinct points, namely $V(g_1)$. This proves the $d$ odd case. If $d$ is even, then $d_1=d_2=\rk(F)$ and $F$ has
infinitely many apolar sets of $\rk(F)$ distinct points. However, for each $P\in\PP^1$ there is a unique set of $\rk(F)$ points (maybe not distinct) apolar to $F$ and
containing $P$. That is, there is a unique element (up to scalar) $g\in (F^\perp)_{d_1}$ vanishing at $P$. Thus, $P\in\caF_F$ if and only if $g$ is not square free.
There are finitely many not square free elements in $(F^\perp)_{d_1}$ since they correspond to the intersection of the line given by $(F^\perp)_{d_1}$ in $\PP(T_{d_1})$
with the hypersurface given by the discriminant; note that the line is not contained in the hypersurface since $(F^\perp)_{d_1}$ contains square free elements.
\end{proof}

\begin{remark}
We can provide a geometric interpretation of Theorem \ref{binaryTHM} for $F$ a degree $d$ binary form of rank $d$, the maximal possible.
In this case, after a change of variables, we can assume $F=xy^{d-1}$. To see geometrically that $[0:1]\in\caF_F$, we consider the point $[y^d]$ on the degree $d$
rational normal curve of $\PP^d$. Note that $[F]$ belongs to the tangent line to the curve in $[y^d]$. Thus, it is easy to see that there does not exist a
hyperplane containing $[F]$ and  $[y^d]$ and cutting the rational normal curve in $d$ distinct points. To prove geometrically that $\caF_F=\{[0:1]\}$ one can
argue using Bertini's theorem. However, for forms of lower rank, we could not find a straightforward geometrical explanation.
\end{remark}

We can improve part (3) of Theorem \ref{binaryTHM} for $d$ even adding a genericity assumption.

\begin{proposition}
Let $d=2h$. If $F\in S_d$ is a generic form of rank $h+1$, then $\caF_F$ is a set of $2h^2$ distinct points.
\end{proposition}
\begin{proof}
Let $\Delta\subset\PP^{h+1}$ be the variety of degree $h+1$ binary forms having at least a factor of multiplicity two. Note that forms having higher degree factors,
or more than one repeated factor, form a variety of codimension at least one in $\Delta$. In particular, a generic line $L$ will meet $\Delta$ in $\deg\Delta$
distinct points each point corresponding to a form of the type $B_1^2B_2\ldots B_{h}$ and $B_i$ is not proportional to $B_j$ if $i\neq j$.

Now, recall the well-known Macaulay's duality between artinian Gorenstein algebras $A_F \simeq S/F^{\perp}$ of socle degree $d$ and homogeneous forms of degree $d$. For a generic $F \in S_d$ we have $F^{\perp} = (g_1,g_2)$ where $\deg (g_1) = \deg (g_2) = h+1$. Therefore, we have that a generic form $F$ determines a generic line in $\PP^{h+1}$ and viceversa. The non square free elements of $(F^{\perp})_{h+1}$ corresponds to $L\cap\Delta$ where $L$ is the line given by $(F^{\perp})_{h+1}$.
By genericity, $L\cap\Delta$ consists of exactly $\deg(\Delta)$ points each corresponding to a degree $h+1$ form $f_i$ having exactly one repeated factor
of multiplicity two. Since every two elements in $(F^\perp)_{h+1}$ have no common factors, it follows that
$$\caF_F=\bigcup_i V(f_i)$$
is a set of $h \deg(\Delta)$ distinct points. Since $\deg(\Delta) = 2h$, the result is now proved.
\end{proof}

We can also iterate the use of Theorem \ref{binaryTHM} to construct a Waring decomposition for a given binary form. Let $F\in S_d$ with rank $r \geq \left\lceil\frac{d+1}{2}\right\rceil$, so that the Waring decomposition is not unique, we can think of constructing such a decomposition one addend at the time.

From our result, we know that in this case the forbidden locus is a closed subset $\caF_F=V(g)$ where $g$ is an element in $F^{\perp}$ of minimal degree; hence,
we can pick any point $[L_1]$ in the open set $\PP^1\setminus V(g)$ to start our Waring decomposition of $F$. This means that there exists $\lambda_1 \in \CC$ such that $F_1=F-\lambda_1L_1^d$ has rank one less than the rank of $F$. If the rank of $F_1$ is still larger than $\left\lceil \frac{d+1}{2} \right\rceil$, we can proceed in the same way as before. We may observe
that $\caF_{F_1} = \caF_{F} \cup [L_1]$. Indeed, by Theorem \ref{binaryTHM}, $\caF_{F_1}=V(g_1)$, where $g_1$ is an element of minimal degree of $F_1^{\perp}$.
Since $\rk(F_1)=\rk(F)-1$, we have that $\deg(g_1)=\deg(g)+1$, in particular it has to be $g_1=gL_1^{\vee}   $, where $L_1^{\vee}$ is the linear differential operator
annihilating $L_1$. Hence, we can continue to construct our decomposition for $F$ by taking any point $[L_2]\in\PP^1\setminus V(g_1)$ and a suitable $\lambda_2\in\CC$ such that $F_2=F-\lambda_1L_1^d-\lambda_2L_2^d$ has rank equal to $\rk(F) - 2$.
We can continue this procedure until we get a form $F_i$ with rank equal to $\left\lceil \frac{d+1}{2} \right\rceil$. If $d$ is odd, we can actually do one more step and arrive to a form $F_i$ with rank strictly less than $\left\lceil \frac{d+1}{2} \right\rceil$. In other words, we have proven the following result.

\begin{proposition}
 Let $F$ be a degree $d$ binary form of rank $r\geq \lceil{d+1 \over 2}\rceil$. For any choice of $[L_1],\ldots,[L_s]\notin\caF_F$ where $s=r-\lceil{d+1 \over 2}\rceil$
 there exists a minimal Waring decomposition for $F$ involving $L_1^d,\ldots,L_s^d$. If $d$ is odd, then it is also unique.
\end{proposition}
\subsection{Plane cubics}\label{Plane Cubics Section}

In this section we describe $\caW_F$ (and $\caF_F$) for $n=2$ and $F\in S_3$, that is for plane cubics. For simplicity, we let $S=\CC[x,y,z]$ and
$T=\CC[X,Y,Z]$.

We use the following characterization of plane cubics adapted from the table given in \cite{LT}.

\begin{center}
\begin{tabular}{c | c c | c || c}
 Type & Description & Normal form & Waring rank & Result \\
\hline
 (1) & triple line & $x^3$ & 1 & Theorem \ref{monomials Thm} \\
 (2) & three concurrent lines &  $xy(x+y)$ & $2$ & Theorem \ref{binaryTHM} \\
 (3) & double line + line  & $x^2y$ & $3$ & Theorem \ref{monomials Thm} \\
 (4) & smooth & $x^3+y^3+z^3$ & $3$ & Theorem \ref{waringdegTHM} \\
 (5) & three non-concurrent lines & $xyz$ & $4$ & Theorem \ref{monomials Thm} \\
 (6) & line + conic (meeting transversally) & $x(yz+x^2)$ & $4$ & Theorem \ref{Fam 6} \\
 (7) & nodal & $xyz-(y+z)^3$ & $4$ & Theorem \ref{family(7)THM} \\
 (8) & cusp & $x^3-y^2z$ & $4$ & Theorem \ref{waringdeggeneralcuspTHM} \\
 (9) & general smooth ($a^3\neq -27,0,6^3$) &$x^3+y^3+z^3+axyz$ & $4$ & Theorem \ref{smoothcubicthm} \\
 (10) & line + tangent conic &$x(xy+z^2)$ & $5$ & Theorem \ref{rank5 cubic} \\
 \hline
&\multicolumn{3}{l}{{\bf Note.} In case (9), $a^3\neq 0,6^3$ so that the rank is actually $4$ and } \\
&\multicolumn{3}{l}{$a^3\neq -27$ for smoothness of the Hessian canonical form \cite{Dol}.}
\end{tabular}
\end{center}

\begin{remark}
We have already analyzed several cases:

 (1),(3),(5): they are monomials and it follows from Theorem \ref{monomials Thm};

 (2): these forms can be seen as forms in two variables, hence it follows from Theorem \ref{binaryTHM}(3);

 (4): smooth plane cubics can be seen as sums of pairwise coprime monomials with high exponents which are analyzed separately
 in the next section, see Theorem \ref{waringdegTHM};

 (8): plane cubic cusps can be seen as the kind of sums of pairwise coprime monomials that we have analyzed in Theorem
 \ref{waringdeggeneralcuspTHM}.
\end{remark}

We now study plane cubics of rank four. First, we need the following lemma.

\begin{lemma}\label{cusp LEMMA}
Let $F$ be a plane cubic of rank four and let $\mathbb{X}$ be a set of four distinct points apolar to $F$. If $\mathbb{X}$ has exactly three collinear points,
then $F$ is a cusp, that is $F$ is of type (8).
\end{lemma}
\begin{proof} We can assume that the three collinear points lie on the line defined by $X$ and the the point not on the line is $[1:0:0]$. Thus,
$XY,XZ\in F^\perp$ and $F=x^3+G(y,z)$. By \cite[Proposition 3.1]{CCC} we have that $\rk(F)=1+\rk(G)$ and thus $\rk(G)=3$. Since all degree three
binary cubics of rank three are monomials we get that, after a change of variables, $G$ can be written as $LM^2$, where $L,M\in \CC[y,z]$ are linear forms. Hence, $F=x^3+LM^2$ and this completes the proof.

\end{proof}

Among the rank 4 plane cubics, we have already analyzed the cusps. Now, we consider families (6),(7) and (9).
Due to Lemma \ref{cusp LEMMA}, we can actually study these families using the approach described in the following remark.

\begin{remark}\label{rankfourstrategy}
Let $F$ be a rank four plane cubic which is not a cusp. Since $F$ is not a binary form, $\caL=(F^{\perp})_2$ is a net of conics and we let
$\caL=\langle C_1,C_2,C_3\rangle$. Since $F$ is not a cusp, all set of four points apolar to $F$ are the complete intersection of two conics.
Thus, when we look for minimal Waring decompositions of $F$, we only need to look at pencils of conics contained in $\caL$ with four distinct base points.

In particular, fixing a point $P\in\PP^2$, we can consider the pencil $\caL(-P)$ of plane conics in $\caL$ passing through
$P$. If $\caL(-P)$ has four
distinct base points, then $P\in\caW_F$; otherwise, we have that the base locus of $\caL(-P)$ is not reduced and $P\in\caF_F$.
In the plane $\PP(\caL)$, we consider the degree three curve $\Delta$ of reducible conics in $\caL$. We recall that a pencil of conics $\caL'$ has
four distinct base points, no three of them collinear, if and only if the pencil contains exactly three reducible conics. In conclusion, given a point $P\in\PP^2$,
we consider the line $\PP(\caL(-P))\subset\PP(\caL)$: if the line is a proper secant line of $\Delta$, that is it cuts $\Delta$ in three distinct points,
we have that $P\in\caW_F$; otherwise, $P\in\caF_F$. Thus we have to study the dual curve $\check\Delta\subset\check\PP(\mathcal{L})$ of lines not intersecting $\Delta$ in three distinct points.

An equation for $\check\Delta$ can be found with a careful use of elimination. To explicitly find $\mathcal{F}_F$ we the consider the map:
\[\phi:\PP(S_1)\longrightarrow\check\PP(\mathcal{L})\]
such that $\phi([a:b:c])=[C_1(a,b,c):C_2(a,b,c):C_3(a,b,c)]$. Note that $\phi$ is defined everywhere and that it is generically $4:1$. In particular, \[\mathcal{F}_F=\phi^{-1}(\check\Delta).\]
\end{remark}

\begin{theorem}\label{Fam 6}
 If $F=x(yz+x^2)$, then $\caF_F=V(XYZ(X^2-12YZ))$.
\end{theorem}
\begin{proof}
Let $\caL=(F^{\perp})_2$ and let $\caC_1: C_1=X^2-6YZ=0$, $\caC_2:C_2=Y^2=0$, and $\caC_3:C_3=Z^2=0$ be the conics generating $\caL$.
In the plane $\PP(\caL)$ with coordinate $\alpha,\beta$ and $\gamma$, let $\Delta$ be the cubic of reducible conics in $\caL$. By computing we get the following equation for $\Delta$:
$$\det
\left[
\begin{matrix}
 \alpha & 0 & 0 \\
 0 & \beta & -3\alpha \\
 0 & -3\alpha & \gamma \\
\end{matrix}
\right]=0,$$

 that is,  $$ \alpha\beta\gamma-9\alpha^3=0.$$

In this case, $\Delta$ is the union of the conic $\caC:9\alpha^2-\beta\gamma=0$ and the secant line $r:\alpha=0$. The line $r$ corresponds to $\caL(-[1:0:0])$ and then, by Remark \ref{rankfourstrategy}, we have that $[1:0:0]\in\caF_F$.

By Remark \ref{rankfourstrategy}, in order to completely describe $\caF_F$, we have to study two family of lines in $\PP(\caL)$: the tangents to the conic $\caC$ and all the lines passing through the intersection points between the line $r$ and the conic $\caC$, that is through the points $[0:0:1]$ and $[0:1:0]$.
More precisely  the point $P= [X:Y:Z]$ is in $ \caF_F$ if  and only if the line $L$ (of the plane $\PP(\caL)$ )
$$L: C_1(P) \alpha+ C_2(P) \beta+ C_3(P) \gamma=0,$$
that is,
$$L: \alpha (X^2-6YZ) + \beta Y^2 + \gamma Z^2=0,$$
falls in one of the following cases:

\begin{enumerate}[(i)]
\item\label{casesix1} $L$ is tangent to the conic $\caC: \beta\gamma-9\alpha^2=0$;

\item\label{casesix2} $L$  passes  through the point $[0:1:0]$;

\item\label{casesix3} $L$  passes  through the point $[0:0:1]$.
\end{enumerate}

In case (\ref{casesix2}) and (\ref{casesix3}) we get that $Y^2=0$ and $Z^2=0$, respectively.
So $V(YZ) \subset \caF_F$.

Now, by assuming $P \notin \{YZ=0\}$. By an easy  computation we get that the line $L$ is tangent to the conic $\caC$ if $X^2(X^2-12YZ)=0$.

It follows that $\caF_F=V(XYZ(X^2-12YZ))$. See Figure \ref{Waring locus F5}.

\begin{figure}[htbp]
\begin{center}
\begin{tikzpicture}
\draw (-3,2) -- (5,2);
\draw (3,4) -- (3,-2);
\draw (-3,4) -- (5,-1.333);

\draw (0,0) ellipse (3cm and 2cm);

\draw (3,2) node {$\bullet$};
\draw (3,0) node {$\bullet$};
\draw (0,2) node {$\bullet$};

\node [xshift=0.6cm, yshift=0.3cm] at (3,2) {[1:0:0]};
\node [above] at (0.3,2) {[0:1:0]};
\node [right] at (3,0.2) {[0:0:1]};

\node [above] at (-1.8,3.5) {$X=0$};
\node [above] at (-2.5,2) {$Z=0$};
\node [right] at (3,-2) {$Y=0$};
\node  at (-3,-2) {$X^2-12YZ=0$};

\end{tikzpicture}
\end{center}
\caption{The forbidden points of $F=x(yz+x^2)$.}
\label{Waring locus F5}
\end{figure}
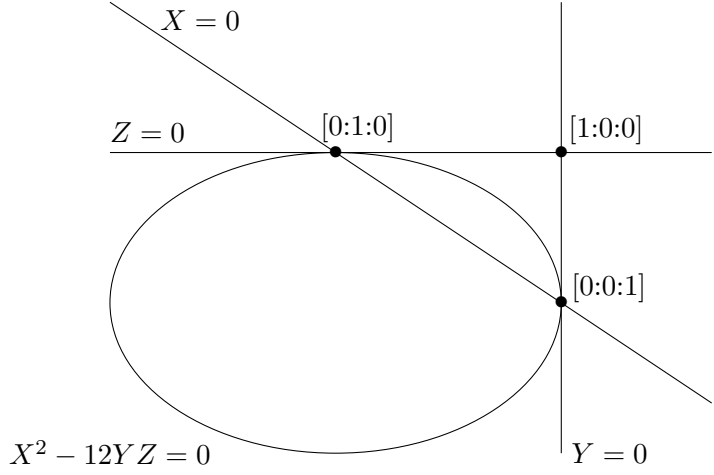

\end{proof}

We now consider family (7), that is nodal cubics.

\begin{theorem}\label{family(7)THM} If $F=y^2z-x^3-xz^2$, then $$\caF_F=V(g_1g_2)$$ where $g_1=X^{3}-6 Y^{2} Z+3 X Z^{2}$ and $g_2=9 X^{4} Y^{2}-4 Y^{6}-24 X Y^{4} Z-30 X^{2} Y^{2} Z^{2}+4 X^{3} Z^{3}-3 Y^{2} Z^{4}-12 X Z^{5}$.
\end{theorem}
\begin{proof}
Note that $[1:0:0]\in\caF_{F}$. In fact $F+x^3=z(y^2-xz)$ represents a conic and a line tangent to it, namely it is in the family $(10)$ and
hence it has rank equal to five.

Let $\caL=(F^{\perp})_2$ and denote by $\caC_1: C_1=XY=0,~\caC_2:C_2=X^2-3Z^2=0$ and $\caC_3:C_3=Y^2+XZ=0$ its generators.

In the plane $\PP(\caL)$ with coordinates $\alpha,\beta$, and $\gamma$ let $\Delta$ be the cubic of reducible conics in $\caL$. By computing we see $\Delta$
that is defined by
$$\det\left[\begin{matrix}\beta & \frac{1}{2}\alpha & \frac{1}{2}\gamma \\ \frac{1}{2}\alpha & \gamma & 0 \\
\frac{1}{2}\gamma & 0 & -3\beta \end{matrix}\right] = 0,$$
that is, $$3\alpha^2\beta - 12\beta^2\gamma - \gamma^3=0.$$

In this case, we have that $\Delta$ is an irreducible smooth cubic. Hence,  we have that
\begin{equation}\label{3points}
 \caF_{F}=\{P\in\PP^2 :\PP(\caL(-P))\text{ is a tangent line to } \Delta\subset\PP(\caL)\}.
\end{equation}
Thus, we are looking for points $P$ such that the line
\[C_1(P)\alpha+C_2(P)\beta+C_3(P)\gamma=0\]
is tangent to $\Delta$. We consider two cases, namely $C_1(P)=0$ and $C_1(P)\neq 0$.

If $C_1(P)\neq 0$, we compute $\alpha$ from the equation of the line and we substitute in the equation of $\Delta$. Then it is enough to compute the discriminant $D$ of the following form in $\beta$ and $\gamma$
\[3(C_2\beta+C_3\gamma)^2\beta-12C_1^2\beta^2\gamma-C_1^2\gamma^3\]
and we get $D=27C_1^4g_1^2g_2$. The script used in \texttt{Macaulay2} to do this computation can be found in \cite[Algorithm 2.3.23]{One16}. Thus, if $C_1(P)\neq 0$, $P\in\caF_F$ if and only if $P\in V(g_1g_2)$.

If $C_1(P)=0$, by direct computation we check that $\caF_F\cap V(C_1)=V(g_1g_2)\cap V(C_1)$. Hence the proof is completed.

\end{proof}

\begin{remark}
In this paper we consider $\caF_F$, and $\caW_F$, as varieties and not as schemes. However, we found in Theorem \ref{family(7)THM} that the ideal of $\caF_F$ is $(g_1^2g_2)$.
\end{remark}

\begin{remark}
 The treatment of the forbidden locus of a plane cubic given in \eqref{3points} is similar to a remark by De
 Paolis recently recalled in \cite{B}. Here, the author gives an algorithm to construct a decomposition of a general plane cubic $F$ as sum of four cubes of linear forms. The algorithm starts with a linear form defining a line intersecting the Hessian of $F$ in precisely three points.
\end{remark}

We now consider the case of cubics in family (9) and we use the map $\phi$ defined in Remark \ref{rankfourstrategy}.

\begin{theorem}\label{smoothcubicthm}
If $F=x^3+y^3+z^3+axyz$ belongs to family $(9)$, then
\begin{enumerate}
\item\label{cubicalisciacasouno} if $\left({a^3-54\over 9a}\right)^3\neq 27$, then $\caF_F=\phi^{-1}(\check\Delta)$ where $\check\Delta$ is the dual curve of the smooth plane cubic $$\alpha^3+\beta^3+\gamma^3-{(a^3-54)\over 9a}\alpha\beta\gamma=0;$$

\item otherwise, $\caF_F$ is the union of three lines pairwise intersecting in three distinct points.
\end{enumerate}
\end{theorem}
\begin{proof}
Let  $\caL=(F^{\perp})_2$ and denote by $\caC_1: C_1=aX^2-6YZ=0$, $\caC_2: C_2=aY^2-6XZ=0,$ and $\caC_3: C_3=aZ^2-6XY=0$ its generators. In the plane $\PP(\caL)$ with
coordinates $\alpha,\beta$, and $\gamma$ let $\Delta$ be the cubic curve of reducible conics. By computing we get an equation for $\Delta$
$$\det
\left[
\begin{matrix}
 a\alpha & -3\gamma & -3\beta \\
 -3\gamma & a\beta & -3\alpha \\
 -3\beta & -3\alpha & a\gamma \\
\end{matrix}
\right]=(a^3-54)\alpha\beta\gamma - 9a\alpha^3-9a\beta^3-9a\gamma^3=0. $$

In the numerical case $\left({a^3-54\over 9a}\right)^3\neq 27$, we have that $\Delta$ is a smooth cubic curve. Thus, we have that
$$\caF_{F}=\{P\in\PP^2 :\PP(\caL(-P))\text{ is a tangent line to } \Delta\subset\PP(\caL)\}.$$
Hence we get $\caF_F$ as described in Remark \ref{rankfourstrategy} using the map $\phi$.

Otherwise, $\Delta$ is the union of three lines intersecting in three distinct points $Q_1,Q_2$ and $Q_3$.  Hence,
$$\caF_{F}=\{P\in\PP^2 :Q_i\in\PP(\caL(-P))\mbox{ for some }i\}$$
and the proof is now completed.
\end{proof}

\begin{example}\label{smoothcubicthm Example}
Consider $a=-6$, thus we are in case \eqref{cubicalisciacasouno} of Theorem \ref{smoothcubicthm}. We can find equations for $\caF_F$ using \texttt{Macaulay2} \cite{M2}. The script used can be found in \cite[Algorithm 2.3.23]{One16}. Hence,
$$\caF_{F} =  V(g_1g_2),$$
where $g_1=X^{3}+Y^{3}-5 X Y Z+Z^{3}$ and
$g_2=27 X^{6}-58 X^{3} Y^{3}+27 Y^{6}-18 X^{4} Y Z-18 X Y^{4} Z-109 X^{2} Y^{2} Z^{2}-58 X^{3} Z^{3}-58 Y^{3} Z^{3}-18 X Y Z^{4}+27 Z^{6}.$

\end{example}

We conclude with family (10), that is cubics of rank five.

\begin{theorem}\label{rank5 cubic}
If $F=x(xy+z^2)$, then $\caF_F=\{[1:0:0]\}$.
\end{theorem}
\begin{proof}
Let $L$ be  a linear form.  The following are equivalent:

\begin{enumerate}[(i)]

\item\label{maxrank1}
$[L]\in\mathcal{F}_F$;

\item\label{maxrank2}
$\rk(F-\lambda L^3)=5$ for all $\lambda\in\mathbb{C}$;

\item\label{maxrank3}
$F-\lambda L^3=0$ is the union of an irreducible conic and a tangent line, for all $\lambda\in\mathbb{C}$;

\item\label{maxrank4}
$F$ and $L^3$ must have the common factor $L$, that is, the line $L=0$ is the line $x=0$.
.
\end{enumerate}

It easy to show that (\ref{maxrank1}) and (\ref{maxrank2}) are equivalent.

For the equivalence between (\ref{maxrank2}) and (\ref{maxrank3}) see the table in Subsection \ref {Plane Cubics Section}.

If  (\ref{maxrank3})  holds, then all the elements in the linear system given by $F$ and $L^3$ are reducible; note that the linear system is not composed with a pencil. Thus, by the second Bertini's Theorem, the linear system has the fixed component $x=0$.

To see that (\ref{maxrank4}) implies (\ref{maxrank3}), note that for all $\lambda\in\mathbb{C}$,  the cubic $x(xy+z^2+\lambda x^2)=0$ is the union of an irreducible conic and a tangent line.
\end{proof}

To finish the treatment of plane cubics we need to deal with the cusp, that is type (4). In three variables an ad hoc argument can be produced, however we prefer to refer to a more general result, namely Theorem \ref{waringdeggeneralcuspTHM}.

\subsection{The forms $x_0^a(x_1^b+\ldots+x_n^b)$ and $x_0^a(x_0^b+x_1^b+\ldots+x_n^b)$}

We study now the Waring loci for a family of reducible forms.

\begin{theorem}\label{reducibleTHM} Let $F= x_0^a(x_1^b+\ldots+x_n^b)$ and
$G= x_0^a(x_0^b+x_1^b+\ldots+x_n^b)$,  where $n \geq 2$, $ a+1 \geq b \geq 3$.
Then
$$\caW_F= \caW_G = V(X_1X_2, X_1X_3,   \ldots , X_1X_n, X_2X_3,  \ldots ,X_2X_n, \ldots , X_{n-1}X_n) \setminus \{P  \}  ,$$
that is, the Waring loci are the coordinate lines through the point $P=[1:0: \ldots :0]$ minus the point itself.

\end{theorem}

\begin{proof}
By  Propositions 4.4 and 4.9 in \cite {CCCGW} we know that

\[ \rk (F) = \rk (G) = (a+1)n.
\]

If  $P\in \caW_F$, we have $\rk (F - \lambda x_0^{a+b}) < (a+1)n$ for some $\lambda  \in \Bbb C$. A contradiction, by  Propositions  4.9 in \cite {CCCGW}.
Hence $P \in \caF_F$.

Analogously, using Proposition 4.4 in \cite {CCCGW}, we get  that $P \in \caF_G$.

Now let $\partial = \alpha _1 X_1+\ldots+ \alpha _n X_n$, where  the $ \alpha _i \in \Bbb C $ are
non-zero, for every $i$.  By Propositions 4.4 and 4.9 in \cite {CCCGW}, we have
\[\rk (\partial  \circ F) = \rk (\partial  \circ G)  =  (a+1)n.
\]

Let $I_{\XX} \subset F^{\perp}$ be the ideal of a set of points giving a Waring decomposition of $F$, i.e. the cardinality
 of $\XX$ is equal to $\rk(F)$. Thus, $I_{\XX^\prime } =  I_{\XX} : (\partial)$ is the ideal of the points of $\XX$ not on the hyperplane
 $\partial = 0$.
 Since   $$I_{\XX^\prime } = I_{\XX} : (\partial) \subset F^{\perp} : (\partial ) = (\partial  \circ F )^{\perp},$$
 we have that $$(a+1)n= \rk (F) = |\XX| \geq |\XX ^\prime | \geq  \rk (\partial  \circ F)=(a+1)n.$$
 It follows that
$\XX$ does not have points on the hyperplane ${\partial=0}$. Thus
\[\caW_F \subseteq V(X_1X_2, X_1X_3,   \ldots , X_1X_n, X_2X_3,  \ldots ,X_2X_n, \ldots , X_{n-1}X_n) \setminus \{P\}  .\]

The opposite inclusion follows from the proof of Proposition 4.4 in \cite {CCCGW}.

Similarly for $G$.
\end {proof}

\section{Strassen's conjecture}\label{Strassen's conjecture}

Fix the following notation:
\[S=\mathbb{C}[x_{1,0} ,\ldots, x_{1,n_1},
\ldots \ldots
, x_{s,0} ,  \ldots, x_{s,n_s} ],\]
\[T=\mathbb{C}[X_{1,0} ,\ldots, X_{1,n_1},
\ldots \ldots
,X_{s,0}   ,\ldots, X_{s,n_s} ].\]

For $i = 1, \dots , s$, we let
\[ S^{[i]} = \mathbb{C}[x_{i,0} ,\ldots, x_{i,n_i}]  ,\]
\[ T^{[i]} =\mathbb{C}[X_{i,0} ,\ldots, X_{i,n_i} ],\]
\[F_i \in  S^{[i]}_d, \]
and
\[ F = F_1 + \cdots + F_s \in S_d .\]

We consider  $F_i \in S$ and thus
\[F_i^\perp=\left\{g\in T \mid g\circ F_i=0\right\}.\]

\begin{conjecture}[{\sc Strassen's conjecture}]\label{strassenconj}
If $F = \sum_{i=1}^s F_i\in S$ is a form such that $F_i \in S^{[i]}$ for all
 $i=1,\ldots,s$, then $$\rk(F)=\rk(F_1)+\ldots +\rk(F_s).$$
\end{conjecture}

\begin{conjecture}\label{strassendecconj}
If $F = \sum_{i=1}^s F_i\in S$ is a degree $d\geq 3$ form such that $F_i \in S^{[i]}$ for all
 $i=1,\ldots,s$, then any minimal Waring decomposition of $F$ is a sum of minimal Waring decompositions of the forms $F_i$.
\end{conjecture}

\begin{remark}
Conjecture \ref{strassendecconj} also appears in \cite{T2016} where, as in Proposition \ref{waringdecPROP}, sufficient conditions are presented which imply the conjecture. The sufficient conditions in \cite{T2016} are different from the one we present in this paper. Moreover, as far as we can tell, the families of Lemma \ref{linearderivativelemma} are not obtained in \cite{T2016}.
\end{remark}

In view of Conjecture \ref{strassendecconj} it is natural to formulate the following conjecture in term of Waring loci. As already explained in Remark \ref{essential variables}, we look at
$\caW_{F_i }\subset \PP^{n_i}_{X_{i,0},\ldots,X_{i,n_i}} \subset \PP^{N}$.

\begin{conjecture}\label{waringlociconj}
If $F = \sum_{i=1}^s F_i\in S$ is a degree $d\geq 3$ form such that $F_i \in S^{[i]}$ for all
 $i=1,\ldots,s$, then $$\caW_F=\bigcup_{i=1,\ldots,r}\caW_{F_i}\subset \PP^N,~\text{ where }N = n_1+\ldots+n_s+s-1.$$
\end{conjecture}

\begin{remark}
 Note that Conjectures \ref{strassendecconj} and \ref{waringlociconj} are false in degree two.  For example, let $F = x^2 - 2yz$. The rank of $F$ is three, but it is easy to find a Waring decomposition of $F$ that is not the sum of $x^2$ plus a Waring decomposition of the monomial $yz$; for example
 $$
  F = (x+y)^2 + (x+z)^2 - (x+y+z)^2.
 $$
\end{remark}

\begin{lemma}
Conjecture \ref{strassendecconj} and Conjecture \ref{waringlociconj} are equivalent and they imply Strassen's conjecture for $d\geq 3$.
\end{lemma}
\begin{proof} Clearly Conjecture \ref{strassendecconj} implies both Conjecture \ref{strassenconj} and Conjecture \ref{waringlociconj}. To complete the proof, we assume that Conjecture \ref{waringlociconj} holds. If $F=\sum_i L_i^d$ is a minimal decomposition of $F$, then each $L_i$ appears in a minimal decomposition of $F_{j(i)}$, thus $L_i$ only involves the variables of $S^{[j(i)]}$. Setting all the variables not in $S^{[j(i)]}$ equal zero in the expression $F=\sum_i L_i^d$, we get a decomposition of $F_{j(i)}$. Note that all the obtained decompositions of the $F_j$ are minimal, otherwise $\rk(F)>\sum_i \rk(F_i)$. Hence Conjecture \ref{strassendecconj} is proved by assuming Conjecture \ref{waringlociconj}.\end{proof}

In order to study our conjectures we prove the following.

\begin{proposition}\label{waringdecPROP}
 Let $F = \sum_{i=1}^s F_i\in S$ be a form such that $F_i \in S^{[i]}$ for all
 $i=1,\ldots,s$. If the following conditions hold

\begin{enumerate}
\item for each $1\leq i\leq s$ there exists a linear derivation $\partial_i\in T^{[i]}$ such that
\[\rk(\partial_i \circ F_i)=\rk(F_i),\]

\item Strassen's conjecture holds for $F_1+\ldots+F_s$,

\item Strassen's conjecture holds for $\partial_1F_1+\ldots+\partial_s F_s$,

\end{enumerate}
then $F$ satisfies Conjecture \ref{waringlociconj}.
\end{proposition}

\begin{proof}
 Let's consider the linear form $t = \alpha_1\partial_1+\ldots+\alpha_s\partial_s$, with $\alpha_i \neq 0$ for all $i=1,\ldots,s$.

 Let $I_{\XX} \subset F^{\perp}$ be the ideal of a set of points giving a minimal Waring decomposition of $F$, i.e. the cardinality
 of $\XX$ is equal to $\rk(F)$. Thus, $I_{\XX} : (t)$ is the ideal of the points of $\XX$ which are outside the linear space
 $t = 0$. We can look at $$I_{\XX} : (t) \subset F^{\perp} : (t) = \left( t \circ F\right)^{\perp}.$$

  By the assumptions we get that $\rk(F) = \rk( t \circ F)$, hence the set of points corresponding to $I_{\XX} : (t)$ has cardinality equal to $\rk (F)$; it follows that $\XX$ does not have points on the hyperplane ${t=0}$.

For the sake of simplicity, if necessary, we rename and reorder the variables in such a way that $\partial_i=X_{i,0}$.

 \medskip
 {\it Claim.} If $P = [a_{1,0}:\ldots:a_{1,n_1}:\ldots:a_{s,0}:\ldots:a_{s,n_s}]$ belongs to $\caW_F$ then in the set $\{a_{1,0},\ldots,a_{s,0}\}$ there is {\it exactly} one non-zero coefficient.
 \medskip

 The claim follows from the first part, since if we have either no or at least two non-zero coefficients in the set $\{a_{1,0},\ldots,a_{s,0}\}$ it is easy to find a linear space $\{ t = 0 \}$ containing the point $P$ and contradicting the assumption
 that it belongs to the Waring locus of $F$.

 Let's consider $\XX_i = \XX \smallsetminus \{x_{i,0} = 0\}$, for all $i=1,\ldots,s$. Similarly as above, by looking at $$I_{\XX_i} = I_{\XX} : (\partial_i) \subset F^{\perp} : (\partial_i) = \left({\partial_i \circ F_i}\right)^{\perp}$$
 we can conclude that the cardinality of each $\XX_i$ is at least $\rk(F_i)$. Moreover, by the claim, we have that the $\XX_i$'s are all pairwise disjoint. By additivity of the rank, we conclude that
 \[\XX = \bigcup_{i=1,\ldots,s} \XX_i,\]
 with $\XX_i\cap\XX_j = \emptyset$, for all $i\neq j$, and
 \[|\XX_i| = \rk(F_i),\]
 for all $i=1,\ldots,s$.

 Hence, we have that the sets $\XX_i$ give minimal Waring decompositions of the forms ${\partial_i \circ F_i}$'s and, by Proposition \ref{prop:essential_variables}, they lie in $\PP^{n_i}_{X_{i,0},\ldots,X_{i,n_i}}$, respectively. Since $\XX$ gives a minimal Waring decomposition of $F$, specializing to zero the variables not in $S^{[i]}$ we see that $\XX_i$ gives a minimal Waring decomposition of $F_i$. Hence, it follows $\caW_F\subset\bigcup_{i=1,\ldots,s}\caW_{F_i}$.

 The other inclusion is easily seen to be true.
\end{proof}

As we show in the following lemma, there are several families of forms for which we can apply Proposition \ref{waringdecPROP}.

\begin{lemma}\label{linearderivativelemma}
If $F$ is one of the following degree $d$ forms

\begin{enumerate}

\item a monomial $x_0^{d_0}\cdot\ldots\cdot x_n^{d_n}$ with $d_i\geq 2$ for $0\leq i\leq n$;

\item a binary form of less than maximal rank, i.e., $F \neq LM^{d-1}$, with $L,M$ linear forms;

\item $x_0^a(x_1^b+\ldots+x_n^b)$ with $n\geq 2$ and $a+1\geq b >2$;

\item $x_0^a(x_0^b+x_1^b+\ldots+x_n^b)$ with $n\geq 2$ and $a+1\geq b >2$;

\item $x_0^aG(x_1,\ldots,x_n)$ with $a\geq 2$ and such that $G^\perp$ is a complete intersection in $\CC[x_1,\ldots,x_n]$ generated by forms of degree at least $a+1$,
\end{enumerate}

then there exists a linear derivation $\partial$ such that
\[\rk(\partial \circ F)=\rk(F).\]
\end{lemma}
\begin{proof}
 {(1)} Let $F = x_0^{d_0}\cdots x_n^{d_n}$ with $d_0 \leq \ldots \leq d_n$, then we know by \cite{CCG}
 that $\rk(F) = (d_1+1)\cdots (d_n+1)$. If we let $\partial = X_0$, then $\rk(F) = \rk(\partial \circ F)$.

 \smallskip
 (2) We know that  $F^{\perp} = (g_1,g_2)$ with $\deg(g_i) = d_i$, $d_1\leq d_2$ and $d_1+d_2 = d+2$. We have to consider different cases.
 \begin{enumerate}
  \item[a)] If $d_1 < d_2$ and $g_1$ is square-free, then $\rk(F) = \deg(g_1)$. Consider any linear form $\partial \in T_1$ which is not a factor of $g_1$.  Then, $(\partial \circ F)^\perp = F^{\perp} : (\partial) = (h_1,h_2)$, with $\deg(h_1)+\deg(h_2) = d+1$. Since $\partial$ is not a factor of $g_1$, then we have that $g_1 = h_1$ and, since it is square-free, we have that $\rk(\partial \circ F) = \rk(F)$.

  \item[b)] If $d_1 < d_2$ and $g_1$ is not square-free, say $g_1 = l_1^{m_1}\cdots l_s^{m_s}$, with $m_1\geq \ldots\geq m_s$ and $m_1\geq 2$, then we have $\rk(F) = \deg(g_2)$. Fix $\partial = l_1\in T_1$. Then, since $F$ is not of the form $LM^{d-1}$, with $L,M\in S_1$, we have that $(\partial \circ F)^{\perp} = F^{\perp} : (l_1) = (h_1,h_2)$ with $\deg(h_1)+\deg(h_2) = d+1$. Since $l_1^{m_1-1}\cdots l_s^{m_s} \in (\partial \circ F)^{\perp}$, but not in $F^{\perp}$ it has to be $h_1 = l_1^{m_1-1}\cdots l_s^{m_s}$. In particular, $\rk(\partial \circ F) = \deg(h_2) =\deg(g_2) =\rk(F)$.

  \item[c)] If $d_1 = d_2$, we can always consider a non square-free element $g \in ( F^{\perp} )_{d_1}$. Indeed, if both $g_1$ and $g_2$ are square-free, then it is enough to consider one element $[g]$ lying on the intersection between the hypersurface in $\PP(S_{d_1})$ defined by the vanishing of the discriminant of polynomials of degree $d_1$ and the line passing through $[g_1]$ and $[g_2]$. Thus, $g = l_1^{m_1}\cdots l_s^{m_s}$, with $m_1\geq \ldots\geq m_s$ and $m_1\geq 2$. Fix $\partial = l_1 \in T_1$. Hence, we conclude similarly as part b).
 \end{enumerate}

 (3) If $F = x_0^a(x_1^b+\ldots+x_n^b)$ with $b,n\geq 2$ and $a+1\geq b$, then we have that $\rk(F) = (a+1)n$,
 by \cite{CCCGW}. If we set $\partial = X_1+\ldots+X_n$, then $\partial \circ F =
 x_0^a(x_1^{b-1}+\ldots+x_n^{b-1})$ and the rank is preserved.

 (4) If $F = x_0^a(x_0^b+\ldots+x_n^b)$ with $b,n\geq 2$ and $a+1\geq b$, then we have that $\rk(F) = (a+1)n$,
 by \cite{CCCGW}. If we set $\partial = X_1+\ldots+X_n$, then $\partial F =
 x_0^a(x_1^{b-1}+\ldots+x_n^{b-1})$ and the rank is preserved.

 (5) If $F = x_0^a G(x_1,\ldots,x_n)$ with $G^\perp=(g_1,\ldots,g_n)$, $a\geq 2$, and $\deg g_i\geq a+1$, we know that
 $\rk(F) = d_1\cdots d_n$, by \cite{CCCGW}. If we consider $\partial = X_0$, then we have that $\partial\circ  F =
 x_0^{a-1}G(x_1,\ldots,x_n)$ and the rank is preserved.
\end{proof}

\begin{theorem}\label{waringdegTHM}
 Let $F = \sum_{i=1}^s F_i\in S$ be a form such that $F_i \in S^{[i]}$ for all
 $i=1,\ldots,s$. If each $F_i$ is one of the following,
\begin{enumerate}
\item a monomial $x_0^{d_0}\cdot\ldots\cdot x_n^{d_n}$ with $d_i\geq 2$ for $0\leq i\leq n$;

\item a binary form of less than maximal rank, i.e., $F \neq LM^{d-1}$, with $L,M$ linear forms;

\item $x_0^a(x_1^b+\ldots+x_n^b)$ with $b,n\geq 2$ and $a+1\geq b$;

\item $x_0^a(x_0^b+x_1^b+\ldots+x_n^b)$ with $b,n\geq 2$ and $a+1\geq b$;

\item $x_0^aG(x_1,\ldots,x_n)$ with $a\geq 2$ and such that $G^\perp$ is a complete intersection in $\CC[x_1,\ldots,x_n]$ generated by forms of degree at least $a+1$,
\end{enumerate}
 then Conjecture \ref{waringlociconj} holds for $F$.
\end{theorem}

\begin{proof}
 We check that conditions $(1)$, $(2)$, and $(3)$ of Proposition \ref{waringdecPROP} are verified.
  Condition $(1)$ holds because of Lemma \ref{linearderivativelemma} and condition $(2)$ follows from Theorem 6.1 in \cite{CCCGW}. To check condition $(3)$ we use the linear derivations $\partial_i$ appearing in the proof of Lemma \ref{linearderivativelemma}. Namely, we note that Theorem 6.1 in \cite{CCCGW} applies to the sum $\sum _i \partial_iF_i$, thus condition $(3)$ follows. The result is now proved.
\end{proof}

\section{More results and open problems}\label{Furtherresultsandopenproblems}

We prove that Conjecture \ref{waringlociconj} holds for the sum of two monomials where one of them has the lowest exponent equal to one. In this case, Lemma \ref{linearderivativelemma} does not apply and we need different methods.

\begin{theorem}\label{waringdeggeneralcuspTHM}
 Conjecture \ref{waringlociconj} is true for a form
 $$
  F = M_1+M_2 = x_{0}x_1^{a_{1}}\cdots x_{n}^{a_{n}} + y_0^{b_0}\cdots y_m^{b_m},\text{ where }\deg F~ = d \geq 3.
 $$
\end{theorem}
We assume that $b_{0} \leq b_{i}$, for any $i = 1,\ldots,m$. Before proving the theorem, we need some preliminary result.

\begin{lemma}\label{perpofsum}
 Let $F = M_1+M_2= x_{0}^{a_{0}}x_1^{a_{1}}\cdots x_{n}^{a_{n}} + y_0^{b_0}\cdots y_m^{b_m}$ be a form of degree $d$.  Then,
 $$
  {\rm length}~T/F^\perp = {\rm length}~T/M_1^\perp + {\rm length}~T/M_2^\perp - 2.
 $$
\end{lemma}
\begin{proof}
We have that $F^{\perp} = M_1^\perp \cap M_2^\perp + \left(G\right)$, with $G = \prod_{i=0}^m(b_i!)X_0^{a_0}\cdots X_n^{a_n}-\prod_{j=0}^n(a_j!)Y_0^{b_0}\cdots Y_m^{b_m},$ (see also \cite[Lemma 1.12]{BBKT15}). Hence, $F^{\perp}$ and $M_1^\perp \cap M_2^\perp$ differ only in degree $d$. Since $G$ is not contained in $M_1^\perp \cap M_2^\perp$, we get
\begin{equation}\label{eq:1.1}
  {\rm length}~T/F^\perp = {\rm length}~T/(M_1^\perp \cap M_2^\perp) - 1.
\end{equation}
Considering the exact sequence
\begin{equation}\label{exact sequence}
0 \longrightarrow  T/(I \cap J )
\longrightarrow T/I \oplus T/J
\longrightarrow T/(I+J)
\longrightarrow 0,
\end{equation}
it follows from \eqref{eq:1.1} that
\begin{equation}\label{eq:1.2}
  {\rm length}~T/F^\perp = {\rm length}~T/M_1^\perp + {\rm length}~T/M_2^\perp  - {\rm length}~T/(M_1^\perp + M_2^\perp) - 1.
\end{equation}
Now, since $Y_{i} \in M_1^{\perp}$ and $X_j \in M_2^\perp$, for all $i=0,\ldots,n,~j = 0,\ldots,m$, we have that $M_1^\perp + M_2^\perp$ is the maximal ideal. Thus, from \eqref{eq:1.2}, we conclude.
\end{proof}

\begin{lemma}\label{rangodiF-2}
 Let $F = M_1+M_2$ be as in Theorem \ref{waringdeggeneralcuspTHM}. Then,
 $$
  {\rm length}~T/(F^\perp:(X_0+Y_0) + (X_0+Y_0))
  = \rk(F) - 2.
 $$
\end{lemma}
\begin{proof}
We have
\begin{align*}
  F^{\perp} :(X_0+Y_0) + (X_0+Y_0)  & = \big((X_0+Y_0)\circ F\big)^{\perp} + (X_0+Y_0) \\
  & = \left(x_1^{a_1}\cdots x_n^{a_n} +
 y_0^{b_0-1} y_1^{b_1}\cdots y_m^{b_m}\right)^{\perp} + (X_0+Y_0)\\
 & = \left(x_1^{a_1}\cdots x_n^{a_n} +
   y_0^{b_0-1} y_1^{b_1}\cdots y_m^{b_m}\right)^{\perp} +(Y_0),
\end{align*}
 where the last equality holds since $ X_0 \in \left(x_1^{a_1}\cdots x_n^{a_n} + y_0^{b_0-1} y_1^{b_1}\cdots y_m^{b_m} \right)^\perp$.
 
 In the case $b_0>1$, by direct computation, we get
  \begin{align*}
 \Big(x_1^{a_1}\cdots x_n^{a_n} &+
   y_0^{b_0-1} y_1^{b_1}\cdots y_m^{b_m}\Big)^{\perp}+(Y_0) = \Big(x_1^{a_1}\cdots x_n^{a_n}  \Big)^{\perp} \cap
   \Big( y_0^{b_0-1} y_1^{b_1}\cdots y_m^{b_m}  \Big)^\perp + \\
  & + \left((b_0-1)!\prod _{i=1}^m(b_i!) X_1^{a_1}\cdots X_n^{a_n} -\prod _{i=1}^n (a_i!) Y_0^{b_0-1} Y_1^{b_1}\cdots Y_m^{b_m} \right)+(Y_0) \\
  & = \big(x_1^{a_1}\cdots x_n^{a_n} \big)^{\perp} \cap
  \big( y_0^{b_0-1} y_1^{b_1}\cdots y_m^{b_m} \big)^\perp
  + \left((b_0-1)!\prod_{i=1}^m(b_i!) X_1^{a_1}\cdots X_n^{a_n} \right)+(Y_0) \\
  & = \big(x_1^{a_1}\cdots x_n^{a_n} \big)^{\perp} \cap
  \left(( y_0^{b_0-1}  y_1^{b_1}\cdots y_m^{b_m})^\perp +(Y_0)\right)
  + \big( X_1^{a_1}\cdots X_n^{a_n} \big),
  \end{align*}
where the last equality holds since $Y_0 \in (x_1^{a_1}\cdots x_n^{a_n})^\perp$.

  Moreover, since $X_1^{a_1}\cdots X_n^{a_n}  \in \big( y_0^{b_0-1} y_1^{b_1}\cdots y_m^{b_m} \big)^\perp$, we get
 \begin{align*}
 \big(x_1^{a_1}\cdots x_n^{a_n} \big)^{\perp} &\cap
  \left(( y_0^{b_0-1}  y_1^{b_1}\cdots y_m^{b_m})^\perp +(Y_0)\right)
  + \big( X_1^{a_1}\cdots X_n^{a_n} \big) \\
 &=
   \left((x_1^{a_1}\cdots x_n^{a_n} )^{\perp}  + (X_1^{a_1}\cdots X_n^{a_n})\right)
  \cap
 \left(( y_0^{b_0-1}  y_1^{b_1}\cdots y_m^{b_m} )^\perp +(Y_0)\right).
 \end{align*}
  So, by the exact sequence \eqref{exact sequence}, we get
 \begin{align*}
 {\rm length} & ~T/ (F^{\perp} :(X_0+Y_0) + (X_0+Y_0)  ) =  \\
 & = {\rm length} \ T/ ( (x_1^{a_1}\cdots x_n^{a_n} )^{\perp}  + (X_1^{a_1}\cdots X_n^{a_n}) ) +
  {\rm length}   \  T/
 (( y_0^{b_0-1}  y_1^{b_1}\cdots y_m^{b_m} )^\perp +(Y_0))  \\
 &~~ - {\rm length}   \  T/ ( (x_1^{a_1}\cdots x_n^{a_n} )^{\perp}  + (X_1^{a_1}\cdots X_n^{a_n})
+
  ( y_0^{b_0-1}  y_1^{b_1}\cdots y_m^{b_m} )^\perp +(Y_0)) \\
  & =  {\rm length} \  T/
 (X_0,X_1^{a_1+1},\ldots,X_n^{a_n+1},Y_0,\ldots,Y_m,X_1^{a_1}\cdots X_n^{a_n}) \\
 & ~~~+ {\rm length} \ T/  (X_0,\ldots,X_n,Y_0,Y_1^{b_1+1},\ldots,Y_m^{b_m+1}) - 1 \\
 & = \prod_{i=1}^n (a_i+1) -1 + \prod_{i=1}^m ( b_i+1) -1 = \rk F -2.
\end{align*}
In case $b_0=1$, since $ F^{\perp} :(X_0+Y_0) + (X_0+Y_0)   = \big(x_1^{a_1}\cdots x_n^{a_n} + y_1^{b_1}\cdots y_m^{b_m}\big)^{\perp} +(X_0,Y_0)  $, by Lemma \ref{perpofsum} we get
$$   {\rm length}   \  T/ (F^{\perp} :(X_0+Y_0) + (X_0+Y_0)  )
 = {\rm length}   \  \widetilde T/ \big(x_1^{a_1}\cdots x_n^{a_n} + y_1^{b_1}\cdots y_m^{b_m}\big)^{\perp}
 $$
 $$  = {\rm length}   \  \widetilde T/ \left(x_1^{a_1}\cdots x_n^{a_n} \right)^{\perp} + {\rm length} \ 
 \widetilde T/ ( y_1^{b_1}\cdots y_m^{b_m})^\perp -2 = \rk F -2,
 $$
 where
 $\widetilde T= \mathbb{C}[ X_1,\ldots X_n, Y_1,  \ldots Y_m ].$
 \end{proof}

\begin{lemma}\label{coprime monomials LEMMA} Let
$F = M_1+M_2$ be as in Theorem \ref{waringdeggeneralcuspTHM}. Then, $$\caW_F \subset \{X_0Y_0 = 0\} \subset \PP^{n+m+1}.$$
\end{lemma}

\begin{proof}

 Let $I_{\XX} \subset F^{\perp}$ be the ideal of a minimal set of apolar points for $F$, thus
  $$|\XX| = \rk F .$$
   It is enough to show that there are no points of $\XX$ lying on the hyperplanes
  $\lambda X_0 + \mu Y_0 = 0$, for $\lambda  \mu\neq 0$. After a change of coordinates, we may assume $\lambda = \mu = 1$.

 We consider $I_{\XX'} = I_{\XX} : (X_0+Y_0)$ the ideal of the set of points in $\XX$ which do not lie on $X_0 + Y_0 = 0$.

 The cardinality of $\XX'$ is at least
  the length of the ring $T/(F^{\perp} :~ (X_0+Y_0) + (X_0+Y_0))$, that is, by Lemma \ref{rangodiF-2},
  \begin{equation}\label{ineq:lowerbound_cardinality}
  |\XX'| \geq \rk F -2.
  \end{equation}
It follows that  on the hyperplane $X_0 + Y_0 = 0$ we have at most two points of $\XX$.

\medskip
  {\it Claim:} In degree $1$, the ideal
$I_{\XX}:(X_0+Y_0) + (X_0+Y_0)$ differs from $F^{\perp}:(X_0+Y_0) + (X_0+Y_0)$.
 \begin{proof}[Proof of Claim.]
 As already computed in the proof of Lemma \ref{rangodiF-2}, we have that $F^{\perp}:(X_0+Y_0)+(X_0+Y_0)$ contains two linear forms, namely $X_0$ and $Y_0$.

 Now, assume that
 $$L = \alpha_0X_0 + \ldots + \alpha_nX_n + \beta_0Y_0 + \ldots + \beta_mY_m \in I_{\XX}:(X_0+Y_0).$$
Thus, we have that $L(X_0+Y_0) \in I_{\XX} \subset F^{\perp}.$

In  case $b_0>1$,  since $X_0^2, X_0Y_0, \ldots,X_0Y_m,  X_1Y_0, \ldots,X_nY_0 \in F^\perp$ we get
  $$(\alpha_1 X_0X_1+  \ldots + \alpha_nX_0X_n+ \beta_0 Y_0^2+
  \beta_1 Y_0Y_1+  \ldots + \beta_mY_0Y_m) \circ F= 0,$$
and from this easily follows that
     $ \alpha_1=\ldots=\alpha_n=\beta_0=\beta_1=\ldots=\beta_m=0$. Hence,  $L=  \alpha_0 X_0 $ and so $\alpha_0 X_0 (X_0+Y_0) \in I_{\XX}$.

Now, consider the hyperplane $Y_0=0$. Again, by direct computation, we get
 \begin{align*}F^{\perp} + (Y_0) & = \big(x_0x_1^{a_1}\cdots x_n^{a_n}\big)^{\perp} \cap  \big(y_0^{b_0}y_1^{b_1}\ldots y_m^{b_m}\big)^\perp + \\
& ~~~\left(\prod_{i=1}^m(b_i!) X_0X_1^{a_1}\cdots X_n^{a_n} - \prod(a_i!) Y_0^{b_0}Y_1^{b_1}\ldots Y_m^{b_m} \right)+ (Y_0) \\
& = \big(x_0x_1^{a_1}\cdots x_n^{a_n}\big)^{\perp} \cap  \big(y_0^{b_0}y_1^{b_1}\ldots y_m^{b_m}\big)^\perp +
(X_0X_1^{a_1}\cdots X_n^{a_n}  )+ (Y_0) \\
& \subseteq \big((x_0x_1^{a_1}\cdots x_n^{a_n})^{\perp}+
 (X_0X_1^{a_1}\cdots X_n^{a_n}  , Y_0)\big)
 \cap  \big((y_0^{b_0}y_1^{b_1}\ldots y_m^{b_m})^\perp +
 (X_0X_1^{a_1}\cdots X_n^{a_n}  , Y_0)\big) \\
 & =\big((x_0 x_1^{a_1}\cdots x_n^{a_n})^{\perp} +  (X_0X_1^{a_1}\cdots X_n^{a_n}  )\big)\cap
\big((y_0^{b_0} y_1^{b_1}\ldots y_m^{b_m})^\perp +(Y_0)\big).
\end{align*}
 Hence, using again the exact sequence \eqref{exact sequence}, we get
 \begin{align*}
 {\rm length} \ T/(F^{\perp} + (Y_0) ) & \geq {\rm length} \ T/\big((x_0 x_1^{a_1}\cdots x_n^{a_n})^{\perp} +  (X_0X_1^{a_1}\cdots X_n^{a_n}  )\big) \\
 & + {\rm length} \ T/ \big((y_0^{b_0} y_1^{b_1}\ldots y_m^{b_m})^\perp +(Y_0)\big)
 -1 \\
 & = 2\prod _{i = 1}^n (a_i+1) +  \prod _{i = 1}^m (b_i+1) -2 \\
 & = \rk~F + \prod _{i = 1}^n (a_i+1) -2 >  \rk~F.
 \end{align*}
 So,
 $${\rm length} \ T/(F^{\perp} + (Y_0)  ) >  \rk F.$$
Hence, $Y_0$ is not a zero divisor for $I_{\XX}$ and there are points of ${\XX}$ lying  on the hyperplane  $Y_0=0$.
Since, by \cite[Remark 3.3]{CCG}, there are no points of ${\XX}$ on the linear space defined by the ideal $(X_0,Y_0)$,
and since $ \alpha_0 X_0(X_0+Y_0) \in I_{\XX}$,
 it follows that $ \alpha_0 =0$. So  $I_{\XX}:(X_0+Y_0) + (X_0+Y_0)$ contains only the linear form $X_0+Y_0$, and thus in  case $b_0>1$ the Claim is proved.

In  case $b_0=1$,
since $X_0^2, X_0Y_0, \ldots,X_0Y_m, Y_0^2, X_1Y_0, \ldots,X_nY_0 \in F^\perp$ we get
  $$(\alpha_1 X_0X_1+  \ldots + \alpha_nX_0X_n+
  \beta_1 Y_0Y_1+  \ldots + \beta_mY_0Y_m) \circ F= 0,$$
and so
     $ \alpha_1=\ldots=\alpha_n=\beta_1=\ldots=\beta_m=0$. Hence,  $L=  \alpha_0 X_0 +\beta_0 Y_0$ and $$L\cdot (X_0+Y_0) =( \alpha_0 X_0 +\beta_0 Y_0) (X_0+Y_0) \in I_{\XX}.$$
  Now, consider the hyperplanes $X_0=0$ and $Y_0=0$.
Again, by direct computation, we get
\begin{align*}
 F^{\perp} + (X_0) & = \big(x_0x_1^{a_1}\cdots x_n^{a_n}\big)^{\perp} \cap  \big(y_0y_1^{b_1}\ldots y_m^{b_m}\big)^\perp \\
 & + \left(\prod_{i=1}^m(b_i!) X_0X_1^{a_1}\cdots X_n^{a_n} - \prod_{i=1}^n(a_i!) Y_0Y_1^{b_1}\ldots Y_m^{b_m} \right)+ (X_0)= \\
&  = \big(x_0x_1^{a_1}\cdots x_n^{a_n}\big)^{\perp} \cap  \big(y_0y_1^{b_1}\ldots y_m^{b_m}\big)^\perp +
(Y_0Y_1^{b_1}\ldots Y_m^{b_m} )+ (X_0) \\
& \subseteq \big((x_0x_1^{a_1}\cdots x_n^{a_n})^{\perp}+
 (Y_0Y_1^{b_1}\ldots Y_m^{b_m} , X_0)\big)
 \cap  \big((y_0y_1^{b_1}\ldots y_m^{b_m})^\perp +
 (Y_0Y_1^{b_1}\ldots Y_m^{b_m} , X_0)\big) \\
 & =\big(x_1^{a_1}\cdots x_n^{a_n}\big)^{\perp}  \cap
\big((y_0y_1^{b_1}\ldots y_m^{b_m})^\perp +
 (Y_0Y_1^{b_1}\ldots Y_m^{b_m} )\big).
\end{align*}
 Hence, using again the exact sequence \eqref{exact sequence}, we get
  \begin{align*}
  {\rm length} \ T/\big(F^{\perp} + (X_0) \big) & \geq {\rm length} \ T/\big(x_1^{a_1}\cdots x_n^{a_n}\big)^{\perp} \\
  &+ {\rm length} \ T/ \big((y_0y_1^{b_1}\ldots y_m^{b_m})^\perp + (Y_0Y_1^{b_1}\ldots Y_m^{b_m} )\big) -1 \\
  & = \prod_{i = 1}^n (a_i+1) + 2 \prod_{i = 1}^m (b_i+1) -2
  = \rk F + \prod _{i = 1}^m (b_i+1) -2 >  \rk F.
  \end{align*}
It follows that
 $${\rm length} \ T/(F^{\perp} + (X_0)  )>  \rk F.$$
 Analogously we have
 $${\rm length} \ T/(F^{\perp} + (Y_0) ) > \rk F.$$

  So $X_0$ and $Y_0$ are not zero divisors for $I_{\XX}$. Hence there are points of ${\XX}$ lying both on the hyperplane $X_0=0$ and on $Y_0=0$.

Since,
by \cite[Remark 3.3]{CCG}, there are no points of ${\XX}$ on the linear space defined by the ideal $(X_0,Y_0)$,
and since $( \alpha_0 X_0 +\beta_0 Y_0) (X_0+Y_0) \in I_{\XX}$,
 it follows that $ \alpha_0 =\beta_0=0$. Thus  $I_{\XX}:(X_0+Y_0) + (X_0+Y_0)$ contains only the linear form $X_0+Y_0$, and the Claim is proved also in case $b_0=1$.
   \end{proof}

 Now, the idea is to show that $I_{\XX}:(X_0+Y_0) + (X_0+Y_0)$ differs from $F^{\perp}:(X_0+Y_0) + (X_0+Y_0)$ also in degree $d-1$. From this, the Claim above and \eqref{ineq:lowerbound_cardinality}, it would follow that the cardinality of $\XX'$ is actually $\rk F$ and then we have no points of
 $\XX$ over the hyperplane $X_0+Y_0 = 0$.

 Consider first the case $b_0 >1$. In this case,
 since (see the proof of Lemma \ref{rangodiF-2})
 $$F^{\perp}~:(X_0+Y_0)+(X_0+Y_0)=
 ( (x_1^{a_1}\cdots x_n^{a_n} )^{\perp}  + (X_1^{a_1}\cdots X_n^{a_n}) )
  \cap
  (( y_0^{b_0-1} y_1^{b_1}\cdots y_m^{b_m} )^\perp + (Y_0)),$$
 we have that $F^{\perp}~:(X_0+Y_0)+(X_0+Y_0)$ contains the whole vector space $T_{d-1}$.
We will prove that
 $(I_{\XX} : (X_0+Y_0) + (X_0+Y_0))_{d-1} \neq T_{d-1}. $
 Since,  from the Claim,
 $I_{\XX}:(X_0+Y_0) + (X_0+Y_0)$ differs from $F^{\perp}:(X_0+Y_0) + (X_0+Y_0)$,
 then
  $$|\XX'| \geq 1+ \hbox{ length } T/(F^{\perp}~:(X_0+Y_0)+(X_0+Y_0) )  = \rk F-1.$$ Hence
 there is at most one point of $\XX$, say $P$,  lying on the hyperplane $X_0 + Y_0 = 0$.
 Since there are no points on the linear space $(X_0,Y_0)$, we can write
 $P = [1,u_1,\ldots,u_n,-1,v_1, \ldots,v_m]$.

  Let
  $$H =X_1^{a_1}\cdots X_n^{a_n}- u_1^{a_1}\cdots u_n^{a_n}X_0^{d-1}.$$
If we assume, by contradiction,   that
$I_{\XX} : (X_0+Y_0) + (X_0+Y_0)$ contains all the forms of degree $d-1$, we have that
$$H  \in I_{\XX} : (X_0+Y_0) + (X_0+Y_0),$$
that is,
$$H + (X_0+Y_0)G \in I_{\XX} : (X_0+Y_0)$$
for
some $G \in T_{d-2}$.
Since $H + (X_0+Y_0)G$ vanishes at $P$ and at the points of $\XX'$, we actually have
$$H + (X_0+Y_0)G \in I_{\XX} \subset F^{\perp},$$
and from this
$$\big(H + (X_0+Y_0)G \big) \circ  F=0.$$
Now, recalling that $d \geq3$, and so $X_0^{d-1} \circ F=0$, we get
$$\big(X_1^{a_1}\cdots X_n^{a_n}- u_1^{a_1}\cdots u_n^{a_n}X_0^{d-1}+ (X_0+Y_0)G\big)\circ F=\prod_{i=1}^n(a_i !) x_0 +G\circ \big(x_1^{a_1}\cdots x_n^{a_n} + y_0^{b_0-1}y_1^{b_1}\ldots y_m^{b_m}\big) =0,
$$
 and this is impossible, since $G\circ (x_1^{a_1}\cdots x_n^{a_n} + y_0^{b_0-1}y_1^{b_1}\ldots y_m^{b_m})$ cannot be a multiple of $x_0$.

 Now, let $b_0=1$.
 In this case we have
 $$F^{\perp}:(X_0+Y_0)+(X_0+Y_0)= \big(x_1^{a_1}\cdots x_n^{a_n} +
  y_1^{b_1}\cdots y_m^{b_m}\big)^{\perp} +(X_0,Y_0), $$
  hence, in degree $d-1$,
  $$\dim \big(F^{\perp}:(X_0+Y_0)+(X_0+Y_0)\big)_{d-1} = \dim  T _{d-1} -1.$$

 Since,  from the Claim,
 $I_{\XX}:(X_0+Y_0) + (X_0+Y_0)$ differs from $F^{\perp}:(X_0+Y_0) + (X_0+Y_0)$,
 then
  $|\XX'| \geq 1+ \hbox{ length } T/(F^{\perp}~:(X_0+Y_0)+(X_0+Y_0) )  = \rk F-1$. Hence
 there is at most one point of $\XX$, say $P$,  lying on the hyperplane $X_0 + Y_0 = 0$.
 Since there are no points on the linear space $(X_0,Y_0)$, we can assume that
  $P = [1,u_1,\ldots,u_n,-1,v_1, \ldots,v_m]$.

  Let
  $$H_1 =X_1^{a_1}\cdots X_n^{a_n}- u_1^{a_1}\cdots u_n^{a_n}X_0^{d-1},$$
  $$H_2 =Y_1^{b_1}\cdots Y_m^{b_m}-(-1)^{d-1}v_1^{b_1}\cdots v_m^{b_m}Y_0^{d-1}.$$
We prove that  $H_1 \notin I_{\XX} : (X_0+Y_0) + (X_0+Y_0)$. In fact, if
$H_1 \in I_{\XX} : (X_0+Y_0) + (X_0+Y_0)$ we have
$$H_1 + (X_0+Y_0)G_1 \in I_{\XX} : (X_0+Y_0)$$
for  some $G_1 \in T_{d-2}$.
But $H_1 + (X_0+Y_0)G_1$ vanishes at $P$ and at the points of $\XX'$,  so we have
$$H_1 + (X_0+Y_0)G_1 \in I_{\XX} \subset F^{\perp},$$
and from this
$$(H_1 + (X_0+Y_0)G_1 ) \circ  F=0.$$
But
\begin{align*}(H_1 + (X_0+Y_0) G_1) \circ F & = \big(X_1^{a_1}\cdots X_n^{a_n}- u_1^{a_1}\cdots u_n^{a_n}X_0^{d-1}+ (X_0+Y_0)G_1\big)\circ F= \\
& =\prod_{i=1}^n (a_i !) x_0 +G_1\circ \big(x_1^{a_1}\cdots x_n^{a_n} + y_0^{b_0-1}y_1^{b_1}\ldots y_m^{b_m}\big) =0,
\end{align*}
 and  this is impossible, since $G_1\circ (x_1^{a_1}\cdots x_n^{a_n} + y_0^{b_0-1}y_1^{b_1}\ldots y_m^{b_m})$ cannot be a multiple of $x_0$.

Now, we can also show that $H_2 \notin I_{\XX} : (X_0+Y_0) + (X_0+Y_0) + (H_1)$. Indeed, assume by contradiction that there exists $\alpha\in\CC$ and $G_2 \in T_{d-2}$ such that
 $$H_2 + \alpha H_1 + (X_0+Y_0)G_2 \in I_{\XX}:(X_0+Y_0).$$
 Since $H_2+\alpha H_1 + (X_0+Y_0)G_2$ vanishes at the point $P$ by construction, we get that
 $$
  H_2+\alpha H_1 + (X_0+Y_0)G_2 \in I_{\XX} \subset F^\perp,
 $$ and, therefore,
 $$
  \big(  H_2+\alpha H_1 + (X_0+Y_0)G_2 \big) \circ F = 0.
 $$
 Hence,
 $$
  \prod_{i=1}^m(b_i!) y_0 + \alpha \prod_{i=1}^n (a_i!) x_0 + G_2 \circ \big( x_1^{a_1}\cdots x_n^{a_n} + y_0^{b_0-1}\cdots y_m^{b_m}\big) = 0.
 $$
  Since $G \circ \big( x_1^{a_1}\cdots x_n^{a_n} + y_0^{b_0-1}\cdots y_m^{b_m}\big)$ cannot produce multiples of $x_0$ and $y_0$, we get a contradiction.

 From this, it follows that
$$\dim \big(I_{\XX} : (X_0+Y_0) + (X_0+Y_0)\big)_{d-1} \leq \dim  T _{d-1} -2.$$
Therefore,
 $$\big(I_{\XX} : (X_0+Y_0) + (X_0+Y_0)\big)_{d-1} \neq \big(F^{\perp}~:(X_0+Y_0)+(X_0+Y_0) \big)_{d-1}.$$
\end{proof}

\begin{lemma}\label{CCG Lemma}
 Let $F = M_1+M_2$ be as in Theorem \ref{waringdeggeneralcuspTHM} and let $\XX$ be a minimal set of points apolar to $F$. Then,
$$
   ( F^\perp : (X_0,Y_0))_2 \subseteq \left(I_{\XX} + (X_0+Y_0)\right)_2.
 $$
\end{lemma}
\begin{proof}
Let  $\XX$ be a minimal set of points apolar to $F$.
Following the proof of \cite[Theorem 3.2]{CCG},  and considering the following chain of inclusions
 \begin{equation}\label{eq:inclusion}
  I_{\XX}:(X_0,Y_0)+(X_0+Y_0) \subseteq F^{\perp}:(X_0,Y_0) + (X_0+Y_0) \subseteq J_1 \cap J_2,
 \end{equation}
 where
 $$
  J_1 = (X_0,X_1^{a_1+1},\ldots,X_n^{a_n+1},Y_0,\ldots,Y_m)
 \ \hbox{and} \ \ J_2 = (X_0,X_1,\ldots,X_n,Y_0,Y_1^{b_1+1},\ldots,Y_m^{b_m+1}),
 $$
 we get
  that $I_{\XX}= I_{\XX}:(X_0,Y_0)$ and
  $$
  \dim_{\CC}\left(I_{\XX}+(X_0+Y_0)\right)_i = \dim_{\CC}\left( J_1 \cap J_2 \right)_i, \text{ for any } i \neq 1;
 $$
 Therefore,  considering the inclusion \eqref{eq:inclusion}, it follows that
  $$
 \left(I_{\XX}+(X_0+Y_0)\right)_i = \left(F^\perp:(X_0,Y_0)+(X_0+Y_0) \right)_i, \text{ for any } i \neq 1;
 $$
 this is enough to conclude the proof.
\end{proof}

\begin{proof}[Proof of Theorem \ref{waringdeggeneralcuspTHM}]
 By Lemma \ref{coprime monomials LEMMA} and  \cite[Remark 3.3]{CCG},
  $\caW_F \subset V(X_0Y_0)\smallsetminus V(X_0,Y_0)$. Hence, we can write any minimal Waring decomposition of $F$ as
 \begin{equation}\label{eq:minimal dec 1}
  F = \sum_{i=1}^{r_1} (L^\prime_i)^d + \sum_{j=1}^{r_2} (L^{\prime\prime}_j)^d,
 \end{equation}
 where, $L_i^\prime\in V(Y_0)$ and $L^{\prime\prime}_j\in V(X_0)$, namely,
 $$L^\prime_i = \alpha_{i,0}x_0+\ldots+\alpha_{i,n}x_n + \beta_{i,1}y_1 + \ldots + \beta_{i,m}y_m,$$
where $\alpha_{i,0}\neq 0$, for $i = 1,\ldots,r_1$, and
$$L^{\prime\prime}_j = \gamma_{j,1}x_1+\ldots+\gamma_{j,n}x_n + \delta_{j,0}y_0 + \ldots + \delta_{j,m}y_m,$$
where $\delta_{j,0} \neq 0$, for $j = 1,\ldots,r_2$.

Let $\XX$ be the set of points apolar to $F$ corresponding to the decomposition \eqref{eq:minimal dec 1}.

 Now, we prove the following statement.

 \smallskip
 {\it Claim:} $L^\prime_i \in V(Y_0,\ldots,Y_m)$, for $i = 1,\ldots,r_1$, and $L^{\prime\prime}_j \in V(X_0,\ldots,X_n)$, for $j = 1,\ldots,r_2$.

 \begin{proof}[Proof of Claim.]Since $L_i^\prime\in V(Y_0)$ and $L^{\prime\prime}_j\in V(X_0)$ and
 since $\text{coeff}_{x_0}(L'_i) = \alpha_{i,0}\neq 0$, for any $i = 1,\ldots,r_1$, and $\text{coeff}_{y_0}(L''_j) = \delta_{j,0} \neq 0$, for any $j=1,\ldots,r_2$, in order to prove the Claim, it is enough to show that $X_0Y_1, \ldots, X_0Y_m$ and $X_1Y_0, \ldots, X_nY_0$ are in $I_{\XX}$.

  We have that $X_0Y_l \in F^{\perp}$; hence, by Lemma \ref{CCG Lemma},
  \begin{align}\label{eq:fact 1a}
  X_0Y_l \in I_{\XX} + (X_0+Y_0) & \Longrightarrow X_0Y_l+ G_l'\cdot (X_0+Y_0) \in I_{\XX}\subset F^{\perp}, (\deg(G_l') = 1) \\
  & \Longrightarrow G_l'\in \left( (X_0+Y_0)\circ F\right)^{\perp}. \nonumber
  \end{align}

  Since $X_kY_0\in F^{\perp}$ ($k=1,\ldots,n$), we get
  \begin{align}\label{eq:fact 1b}
  X_kY_0 \in I_{\XX} + (X_0+Y_0) & \Longrightarrow X_kY_0 + G_k''\cdot (X_0+Y_0) \in I_{\XX}\subset F^{\perp}, (\deg(G_k'') = 1) \\
  & \Longrightarrow G_k''\in \left( (X_0+Y_0)\circ F\right)^{\perp}. \nonumber
  \end{align}
  Now, we need to distinguish between two cases.

  \smallskip
   If $b_0 \geq 2$, then $\left( (X_0+Y_0)\circ F\right)^{\perp} = \left( x_1^{a_1}\cdots x_n^{a_n}+y_0^{b_0-1}y_1^{b_1}\cdots y_m^{b_m}\right)^\perp.$ Therefore, since $\deg G_l'=\deg G_k''=1$, we have
  $$
   G_l' = a'_lX_0 \text{ and } G''_k = a_k''X_0,\text{ for some } a'_l, a''_k \in \CC.
  $$
 So, from \eqref{eq:fact 1b}, we have that $ X_kY_0 + a''_kX_0(X_0+Y_0) \in I_{\XX}$, and,
multiplying by $X_0$, since $X_0Y_0\in I_{\XX}$, we get
  $$
 a''_kX_0^3 \in I_{\XX}, \text{ for } k = 1,\ldots,n.
  $$
  Since the coefficients of $x_0$ in the $L'_i$'s are different from $0$, it follows that $a''_k = 0$, for all $k=1,\ldots,n$. Therefore, $X_kY_0 \in I_{\XX}$, for all $k = 1,\ldots,n$. Hence, we get that $L''_j\in V(X_0,\ldots,X_n)$, for $j = 1,\ldots,r_2$, and we have
  $$L^{\prime\prime}_j =  \delta_{j,0}y_0 + \ldots + \delta_{j,m}y_m.$$

  From \eqref{eq:fact 1a},  since $X_0Y_0\in I_{\XX}$, we obtain that
    $$
   X_0Y_l + a'_lX_0(X_0 + Y_0) \in I_{\XX} \Longrightarrow X_0(Y_l+a'_lX_0) \in I_{\XX}.
  $$
    Consider the points $P_i=( \alpha_{i,0}, \ldots,\alpha_{i,n}, 0, \beta_{i,1},  \ldots , \beta_{i,m} )\in \XX$ ($i = 1,\ldots,r_1$) associated to the $L^\prime$'s. Since
  $  X_0(Y_l+a'_lX_0) \in I_{\XX} $ and
    $\alpha_{i,0} \neq 0$, for any $i = 1,\ldots,r_1$,
   it follows that
   $$b_{i,l} = -a'_l \alpha_{i,0}, \hbox { for  all } i = 1,\ldots,r_1 \hbox {  and }  l = 1,\ldots,m. $$
    Hence
     $$ \sum_{i=1}^{r_1} \left(L_i^\prime\right)^d  =
    \sum_{i=1}^{r_1} \big(\alpha_{i,0}(x_0-a'_1 y_1 +\ldots - a'_m y_m)+\alpha_{i,1}x_1\ldots+\alpha_{i,n}x_n\big)^d.$$
    Now, observe that, setting $y_0=\ldots=y_m=0$ in \eqref{eq:minimal dec 1}, we obtain that
  $$
  x_0x_1^{a_1}\cdots x_n^{a_n} = \sum_{i=1}^{r_1} (\alpha_{i,0}x_0+\ldots+\alpha_{i,n}x_n)^d;
  $$
so, replacing $x_0$ with $(x_0-a'_1 y_1 +\ldots - a'_m y_m)$ we get that
  \begin{align*}
   \sum_{i=1}^{r_1} \left(L_i^\prime\right)^d & =
    \sum_{i=1}^{r_1} \big(\alpha_{i,0}(x_0-a'_1 y_1 +\ldots - a'_m y_m)+\alpha_{i,1}x_1\ldots+\alpha_{i,n}x_n\big)^d \\ & = (x_0-a'_1 y_1 +\ldots - a'_m y_m)x_1^{a_1}\cdots x_n^{a_n}.
  \end{align*}
  It follows that
  $$F= x_0x_1^{a_1}\cdots x_n^{a_n} + y_0^{b_0}y_1^{b_1}\cdots y_m^{b_m}
  =  (x_0-a'_1 y_1 +\ldots - a'_m y_m)x_1^{a_1}\cdots x_n^{a_n} + \sum_{j=1}^{r_2} (\delta_{j,0}y_0 + \ldots + \delta_{j,m}y_m)^d.$$
Therefore, looking at the coefficient of the monomial $x_1^{a_1}\cdots x_n^{a_n}y_l$,   we obtain $ -a'_l=0$, for any $l = 1,\ldots,m$. Hence from  \eqref{eq:fact 1a},  we get
$X_0Y_l\in I_{\XX}$ for any $l = 1,\ldots,m$.

    \smallskip
    If $b_0 = 1$, then $\left( (X_0+Y_0)\circ F\right)^{\perp} = \left( x_1^{a_1}\cdots x_n^{a_n}+y_1^{b_1}\cdots y_m^{b_m}\right)^\perp.$ Therefore, from   \eqref{eq:fact 1a} and   \eqref{eq:fact 1b} we have
  $$
   G'_l = a'_lX_0+b'_lY_0 \text{ and } G''_k = a''_kX_0 + b''_kY_0,\text{ for some }a'_l,b'_l, a''_k,b''_k \in \CC.
  $$
  Since $X_0Y_0 \in I_{\XX}$, we have, for any $l = 1,\ldots,m$,
  \begin{align*}
   X_0Y_l + (a'_lX_0+b'_lY_0)(X_0+Y_0)\in I_{\XX} & \Longrightarrow X_0Y_l + a'_lX_0^2 + b'_lY_0^2 \in I_{\XX} \\
   & \Longrightarrow X_0^2Y_l + a'_lX_0^3 = X_0^2(Y_l+a'_lX_0) \in I_{\XX}.
  \end{align*}

  Considering the points $P_i=( \alpha_{i,0}, \ldots,\alpha_{i,n}, 0, \beta_{i,1},  \ldots , \beta_{i,m} )\in \XX$ ($i = 1,\ldots,r_1$) associated to the $L^\prime$'s, since $a_{i,0}\neq 0$,  it follows that $b_{i,l} = -a'_l\alpha_{i,0}$, for all $i = 1,\ldots,r_1$ and $l=1,\ldots,m$.

  Analogously, considering the monomials $X_kY_0$, for any $k =1,\ldots,n$, we get $\gamma_{i,k} = -b''_{k}\delta_{i,0}$.

Hence, recalling that $L^\prime_i = \alpha_{i,0}x_0+\ldots+\alpha_{i,n}x_n + \beta_{i,1}y_1 + \ldots + \beta_{i,m}y_m,$ we get
  \begin{align}\label{eq:fact 3}
   \sum_{i=1}^{r_1}(L'_i)^d & = \sum_{i=1}^{r_1}\big(\alpha_{i,0}(x_0-a'_1y_1+\ldots-a'_my_m)+\ldots+\alpha_{i,n}x_n\big)^d .
  \end{align}
  Now, by setting $y_0 = \ldots = y_m = 0$ in \eqref{eq:minimal dec 1}, we obtain
  $$
 x_0x_1^{a_1}\cdots x_n^{a_n} =  \sum_{i=1}^{r_1}(\alpha_{i,0}x_0+\ldots+\alpha_{i,n}x_n)^d + \sum_{j=1}^{r_2}(\gamma_{j,1}x_1 + \ldots + \gamma_{j,n}x_n)^d.
  $$
  By \eqref{eq:fact 3}  and this equality, replacing $x_0$ with $(x_0-a'_1y_1+\ldots-a'_my_m)$,
  we obtain
  $$ \sum_{i=1}^{r_1}(L'_i)^d  = (x_0-a'_1y_1+\ldots-a'_my_m)x_1^{a_1}\cdots x_n^{a_n} - \sum_{j=1}^{r_2}(\gamma_{j,1}x_1 + \ldots + \gamma_{j,n}x_n)^d.
  $$
Analogously, we get
  $$
   \sum_{i=1}^{r_1}(L''_i)^d  = (y_0-b''_1x_1+\ldots-b''_nx_n)y_1^{b_1}\cdots y_m^{b_m} - \sum_{i=1}^{r_1}(\beta_{i,1}y_1 + \ldots + \beta_{i,m}y_m)^d.
   $$
 Thus,
  \begin{align*}
   F = (x_0-a'_1y_1+\ldots-a'_my_m)x_1^{a_1}\cdots &x_n^{a_n} + (y_0-b''_1x_1+\ldots-b''_nx_n)y_1^{b_1}\cdots y_m^{b_m} \\ &- \sum_{i=1}^{r_1}(\beta_{i,1}y_1 + \ldots + \beta_{i,m}y_m)^d - \sum_{j=1}^{r_2}(\gamma_{j,1}x_1 + \ldots + \gamma_{j,n}x_n)^d.
  \end{align*}
  Looking at the coefficient of the monomial $x_1^{a_1}\cdots x_n^{a_n}y_l$ $(l = 1,\ldots,m)$, we get $a'_l = 0$. Similarly, considering the coefficient of the monomial $x_ky_1^{b_1}\cdots y_m^{b_m}$ $(k = 1,\ldots,n)$ we get $b''_k = 0$.

 Now, substituting $a'_l = 0$ in \eqref{eq:fact 1a}, we get that
 $$
  X_0Y_l + b'_lY_0^2 \in I_\XX;
 $$
 multiplying by $Y_0$, since $X_0Y_0\in I_{\XX}$, we have that
 $
  b'_lY_0^3 \in I_\XX, \text{ for any } l = 1,\ldots,m.
 $
 Since the points corresponding to the $L''$'s have the coefficient of $y_0$ different from $0$, it follows that $b'_l = 0$ and, therefore, $X_0Y_j \in I_{\XX}$. 

 Analogously, we can prove that $a''_k = 0$, for all $k = 1,\ldots,n$ and, consequently, that $X_iY_0 \in I_{\XX}$.

 This concludes the proof of the Claim.
 \end{proof}

 So we have proved that
 $$L^\prime_i = \alpha_{i,0}x_0+\ldots+\alpha_{i,n}x_n, \    \ \ \
L^{\prime\prime}_j =  \delta_{j,0}y_0 + \ldots + \delta_{j,m}y_m,$$
 and so
\begin{equation}\label{eq:minimal dec 2}
M_1 =x_0x_1^{a_1}\cdots x_n^{a_n} = \sum_{i=1}^{r_1}  L^\prime_i,\ \ \ \
 M_2 = y_0^{b_0}y_1^{b_1}\cdots y_m^{b_m}=\sum_{i=1}^{r_2}L^{\prime\prime}_j; \end{equation}
therefore, $r_i \geq \rk (M_i)$, for $i = 1,2$. Since $r_1 + r_2 = \rk(F) = \rk (M_1) + \rk (M_2)$, we have that actually $r_i = \rk(M_i)$, for $i = 1,2$. Thus, the expressions in \eqref{eq:minimal dec 2} are minimal decompositions of $M_1$ and $M_2$, respectively. It follows that $\caW_F\subset\caW_{M_1} \cup \caW_{M_2}$. Since the opposite inclusion is trivial, we conclude the proof.

\end{proof}

\subsection{Open problems} There are several aspect that we consider to be worth of further investigation. Here we list some example.

\medskip
(1) It is not clear to us whether some of our results hold when weaker assumptions are made. Consider, for example, Theorem \ref{waringdegTHM} for {\it any} monomial and not only when all exponents are at least two. Similarly, we do not know if Theorem \ref{reducibleTHM} holds in the case of $b=2$.

\medskip
(2) In all of our results $\caF_F$ is never empty and we conjecture that this is the case for any form $F$. Roughly speaking, the more Waring decompositions of $F$ we have, the smaller $\caF_F$ should be. Thus, this last conjecture, seems to have a strong relation with forms of high Waring rank which usually have many sum of powers decompositions. Note, for example in the case of plane cubics, that the smallest $\caF_F$ occurs for $F$ of type (10), that is for plane cubics of {\it maximal} Waring rank.

\medskip
(3) It would be interesting to consider the sets $\caW^t_F$ formed by $t$-uples of linear forms appearing in the same minimal Waring decomposition of $F$. Note that $\caW^{1}_F=\caW_F$, while $\caW^{\rk(F)}_F$ is the well known variety of sum of powers ${\it VSP}(F)$ defined in \cite{RS}.

\medskip
(4) When we are able to find an ideal defining $\caF_F$ this ideal is often not radical. Thus, the scheme of forbidden points comes with a non reduced structure, but it is not clear what this is telling us about $F$.

\end{document}